\documentclass[11pt]{amsart}

\usepackage{mathrsfs,amssymb,amsmath}
\usepackage{verbatim}
\usepackage{hyperref}
\usepackage{graphicx}
\usepackage{enumerate}

\usepackage{vmargin}

\usepackage{subfigure}

\usepackage{layout}
\usepackage{booktabs}

\begin{document}


\newtheorem{theorem}{Theorem}
\newtheorem{proposition}{Proposition}
\newtheorem{lemma}{Lemma}[section]
\newtheorem{corollary}{Corollary}[section]
\newtheorem{conjecture}[theorem]{Conjecture}
\newtheorem{question}[theorem]{Question}
\newtheorem{problem}[theorem]{Problem}

\theoremstyle{definition}
\newtheorem{definition}[theorem]{Definition}
\newtheorem{example}[theorem]{Example}

\newtheorem{remark}{Remark}[section]

\def\theenumi{\roman{enumi}}

\numberwithin{equation}{section}

\renewcommand{\Re}{\operatorname{Re}}
\renewcommand{\Im}{\operatorname{Im}}

\newcommand{\ind}{1\hspace{-.27em}\mbox{\rm l}}
\newcommand{\inde}{1\hspace{-.23em}\mathrm{l}}
\newcommand{\lqn}[1]{\noalign{\noindent $\displaystyle{#1}$}}
\newcommand{\bm}[1]{\mbox{\boldmath{$#1$}}}
\newcommand{\eps}{\varepsilon}
\newcommand{\lip}{{\rm Lip}}
\newcommand{\R}{\mathbb{R}}
\newcommand{\N}{\mathbb{N}}
\newcommand{\Z}{\mathbb{Z}}
\newcommand{\Sym}{\mathfrak{S}}
\newcommand{\Leb}{{\lambda}}
\newcommand{\md}{\mathbb{D}}
\newcommand{\me}{\mathbb{E}}
\newcommand{\Q}{\mathbb{Q}}
\newcommand{\1}{{\sf 1}}
\newcommand{\tmu}{\tilde{\mu}}
\newcommand{\tnu}{\tilde{\nu}}
\newcommand{\hmu}{\hat{\mu}}
\newcommand{\mb}{\bar{\mu}}
\newcommand{\bnu}{\overline{\nu}}
\newcommand{\tx}{\tilde{x}}
\newcommand{\tz}{\tilde{z}}
\newcommand{\A}{{\mathcal A}}
\newcommand{\skria}{{\mathcal A}}
\newcommand{\skrib}{{\mathcal B}}
\newcommand{\skric}{{\mathcal C}}
\newcommand{\skrid}{{\mathcal D}}
\newcommand{\skrie}{{\mathcal E}}
\newcommand{\skrif}{{\mathcal F}}
\newcommand{\skrig}{{\mathcal G}}
\newcommand{\skrih}{{\mathcal H}}
\newcommand{\skrii}{{\mathcal I}}
\newcommand{\skrik}{{\mathcal K}}
\newcommand{\skril}{{\mathcal L}}
\newcommand{\skrim}{{\mathcal M}}
\newcommand{\skrin}{{\mathcal N}}
\newcommand{\skrip}{{\mathcal P}}
\newcommand{\skriq}{{\mathcal Q}}
\newcommand{\skris}{{\mathcal S}}
\newcommand{\skrit}{{\mathcal T}}
\newcommand{\skriu}{{\mathcal U}}
\newcommand{\skriw}{{\mathcal W}}
\newcommand{\skrix}{{\mathcal X}}

\newcommand{\al}{{\alpha}}
\newcommand{\be}{{\beta}}
\newcommand{\ga}{{\gamma}}
\newcommand{\de}{{\delta}}

\newcommand{\heap}[2]{\genfrac{}{}{0pt}{}{#1}{#2}}
\newcommand{\sfrac}[2]{ \, \mbox{$\frac{#1}{#2}$}}
 \renewcommand{\labelenumi}{(\roman{enumi})}
\newcommand{\ssup}[1] {{\scriptscriptstyle{({#1}})}}

\newcommand{\cal}{\mathcal}
\allowdisplaybreaks


\def \R {{\mathbb R}}
\def \HH {{\mathbb H}}
\def \N {{\mathbb N}}
\def \C {{\mathbb C}}
\def \Z {{\mathbb Z}}
\def \Q {{\mathbb Q}}
\def \TT {{\mathbb T}}
\newcommand{\T}{\mathbb T}
\def \Dc {{\mathcal D}}

\newcommand{\tr}[1] {\hbox{tr}\left( #1\right)}

\newcommand{\area}{\operatorname{area}}

\newcommand{\Norm}{\mathcal N}
\newcommand{\simgeq}{\gtrsim}%
\newcommand{\simleq}{\lesssim}

\newcommand{\length}{\operatorname{length}}

\newcommand{\curve}{\mathcal C} 
\newcommand{\vE}{\mathcal E} 
\newcommand{\Ec}{\mathcal {E}} 
\newcommand{\Sc}{\mathcal{S}} 

\newcommand{\dist}{\operatorname{dist}}
\newcommand{\supp}{\operatorname{supp}}
\newcommand{\spec}{\operatorname{spec}}
\newcommand{\diam}{\operatorname{diam}}

\newcommand{\Ccap}{\operatorname{Cap}}
\newcommand{\E}{\mathbb E}

\newcommand{\sumstar}{\sideset{}{^\ast}\sum}

\newcommand {\Zc} {\mathcal{Z}} 
\newcommand{\ninumber}{\Zc}

\newcommand{\zeigen}{E} 
\newcommand{\eigen}{m}

\newcommand{\ave}[1]{\left\langle#1\right\rangle} 

\newcommand{\Var}{\operatorname{Var}}
\newcommand{\Prob}{\operatorname{Prob}}

\newcommand{\var}{\operatorname{Var}}
\newcommand{\Cov}{{\rm{Cov}}}
\newcommand{\meas}{\operatorname{meas}}

\newcommand{\leg}[2]{\left( \frac{#1}{#2} \right)} 

\renewcommand{\^}{\widehat}

\newcommand {\Rc} {\mathcal{R}}

\title[Nodal area distribution for arithmetic random waves]
{Nodal area distribution\\for arithmetic random waves}
\author{Valentina Cammarota}

\address{Department of Mathematics, King's College London \& Dipartimento di Scienze Statistiche, Universit\`a degli Studi di Roma ``La Sapienza"  } 
\email{valentina.cammarota@uniroma1.it}


\begin{abstract}
We obtain the limiting distribution of the nodal area of random Gaussian Laplace eigenfunctions on 
$\mathbb{T}^3= \mathbb{R}^3/ \mathbb{Z}^3$ ($3$-dimensional `arithmetic random waves').
We prove that, as the multiplicity of the eigenspace goes to infinity, the nodal area converges to a universal, non-Gaussian, distribution. Universality follows from the equidistribution of lattice points on the sphere.
Our arguments rely on the Wiener chaos expansion of the nodal area: we show that, analogous to \cite{MPRW}, the fluctuations are dominated by the fourth-order chaotic component. The proof builds upon recent results in \cite{B&M} that establish an upper bound for the number of non-degenerate correlations of lattice points on the sphere. 
We finally discuss higher-dimensional extensions of our result. 
\end{abstract}

\maketitle

\section{Introduction and framework}

\subsection{Toral eigenfunctions and nodal volume}

Let $f: \mathbb{T}^d= \mathbb{R}^d/ \mathbb{Z}^d \to \mathbb{R}$, $d \ge 2$, be the real-valued functions satisfying the eigenvalue equation 
\begin{equation} \label{Schrodinger}
\Delta f + E f=0,
\end{equation}
where $ E>0$ and $\Delta$ is the Laplace-Beltrami operator on $\mathbb{T}^d$; the spectrum of $\Delta$ is totally discrete. 

The nodal set of a function is the zero set. 
Nodal sets for eigenfunctions of the Laplacian on smooth compact Riemannian manifolds have been studied intensively; it is known \cite{cheng} that except for a subset of lower dimension, the nodal sets of eigenfunctions are smooth manifolds of codimension one in the ambient manifold, and hence the {\it nodal volume} $${\rm Vol}(f^{-1}(0))$$ of $f$ is well defined. A fundamental conjecture of Yau \cite{cheng1, cheng2} asserts that for any smooth compact Riemannian manifold ${\cal M}$ there exist constants $0<c_1({\cal M}) \le c_2({\cal M})$ such that 
\begin{align} \label{yau}
c_1({\cal M})\; \sqrt{E} \le {\rm Vol}(f^{-1}(0)) \le c_2({\cal M}) \; \sqrt{E}.
\end{align}
Yau's conjecture was proven for real-analytic metrics by Donnelly and Fefferman \cite{D&F}, the lower bound in Yau's conjecture was recently established for general smooth manifolds by Logunov \cite{logunov}.   
 
For ${\cal M}={\mathbb T}^d$, the eigenspaces of the Laplacian are related to the theory of lattice points on $(d-1)$-dimensional spheres. Let 
$$S=\{n\in {\mathbb Z}: n= n_1^2+ \cdots+n_d^2, \;\; \text{for}\;\; n_1, \dots, n_d \in \mathbb{Z}\},$$ 
be the collection of all numbers expressible as a sum of $d$ squares. The sequence of eigenvalues, or {\it energy levels}, of \eqref{Schrodinger} are all numbers of the form 
$$E_n= 4 \pi^2 n, \hspace{1cm} n \in S.$$ In order to describe the Laplace eigenspace corresponding to $E_n$, we introduce the set of frequencies $\Lambda_n$; for $n \in S_n$ let
$$\Lambda_n= \{\lambda \in \mathbb{Z}^d:\; ||\lambda||^2=n\}.$$
 $\Lambda_n$ is the frequency set corresponding to $E_n$.  Using the notation 
$e(x)=\exp(2 \pi i x),$
the $\mathbb{C}$-eigenspace ${\cal E}_n$ corresponding to $E_n$ is spanned by the $L^2$-orthonormal set of functions $\{e(\langle \lambda, \cdot \rangle) \}_{\lambda \in \Lambda_n}$. We denote the dimension of ${\cal E}_n$ 
$$\cal{N}_n=\text{dim}\, {\cal E}_n=|\Lambda_n|,$$
that is equal to the number of different ways $n$ may be expressed as a sum
of $d$ squares.

\subsection{Arithmetic random waves}

The frequency set $\Lambda_n$ can be identified with the set of lattice points lying on a $(d-1)$-dimensional sphere with radius $\sqrt n$, 
the sequence of spectral multiplicities $\{{\cal N}_n\}_{n \ge 1}$ is unbounded. It is natural to consider properties of  {\it generic} or {\it random} eigenfunctions $f_n \in {\cal E}_n$,  is the high energy asymptotics.  
More precisely, let $f_n: \mathbb{T}^d \to \mathbb{R}$ be the Gaussian random field of (real valued)  ${\cal E}_n$-functions with eigenvalue $E_n$, i.e. the random linear combination  
\begin{align} \label{fn}
f_n(x)=\frac{1}{\sqrt{ \cal{N}_n}} \sum_{\lambda \in \Lambda_n} a_{\lambda} e(\langle \lambda, x \rangle),
\end{align}
where the coefficients $\{a_{\lambda}\}$ are complex-Gaussian random variables verifying the following properties: 
\begin{enumerate}[i)]
\item every $a_{\lambda}$ has the form  $a_{\lambda}={\rm Re}(a_{\lambda})+ i \, {\rm Im}(a_{\lambda})$ where ${\rm Re}(a_{\lambda})$ and ${\rm Im}(a_{\lambda})$  are two independent real-valued, centred, Gaussian random variables with variance $1/2$, \label{i}
\item the $a_{\lambda}$'s  are stochastically independent, save for the relations $a_{-\lambda} = \overline{a}_{\lambda}$ making $f_n$ real-valued. \label{iii}
\end{enumerate}
\vspace{0.2cm}
\noindent By definition, $f_n$ is stationary, i.e. the law of $f_n$ is invariant under all translations 
 $$f(\cdot) \to f(x+\cdot), \hspace{1cm} x \in {\mathbb T}^d;$$ in fact $f_n$ is a centred Gaussian random field with covariance function 
 \begin{align*}
\mathbb{E}[f_n(x) f_n(y)]
 &= \frac{1}{ \cal{N}_n} \sum_{\lambda  \in \Lambda_n}   \cos(2 \pi \langle \lambda, x-y \rangle).
 \end{align*}
Note that the normalising factor in \eqref{fn} is chosen so that $f_n$ has unit-variance.

 \subsection{Prior work on this model}
 
Our object of study is the {\it nodal volume}, i.e. the sequence $\{{\cal V}_n\}_{n \in S}$ of all random variables of the form 
\begin{align*}
{\cal V}_n= \text{Vol} (f^{-1}_n(0)).
\end{align*}
The expected value of ${\cal V}_n$ was computed in \cite{R&W} to be, for every $d \ge 2$,  
\begin{align*}
\mathbb{E}[ {\cal V}_n]= {\cal I}_d \sqrt{\frac{ E_n}{4 \pi^2}}, \hspace{0.5cm}  {\cal I}_d=\sqrt{\frac{4 \pi}{d}} \frac{\Gamma(\frac{d+1}{2})}{\Gamma(\frac d 2)},
\end{align*}
 in agreement with Yau's conjecture \eqref{yau}. The more challenging question of the asymptotic behavior of the variance was also addressed in \cite{R&W} where the following bound for the variance was computed for every dimension $d \ge 2$  
\begin{align*}
\text{Var} ({\cal V}_n)=O\left( \frac{E_n}{\sqrt{ {\cal N}_n}} \right),
\end{align*}
and it was conjectured that the stronger bound 
\begin{align} \label{10luglio}
\text{Var} ({\cal V}_n)=O\left( \frac{E_n}{{ {\cal N}_n}} \right)
\end{align}
should hold. 

\subsubsection{Nodal length}
In \cite{KKW} it was derives the precise asymptotic behavior of the variance of the nodal length  ${\cal L}_n$ of the random eigenfunctions on $\mathbb{T}^2$. 
For $d=2$ the set $\Lambda_n$ induces a discrete probability measure $\mu_n$ on the unit circle ${\cal S}^1=\{z \in \mathbb{C}: |z|=1\}$ by defining 
\begin{align*}
\mu_n=\frac{1}{{\cal N}_n} \sum_{\lambda \in {\cal N}_n} \delta_{\frac{\lambda}{\sqrt n}},
\end{align*} 
where $\delta_x$ is the Dirac delta centred at $x \in {\cal S}^1$. 
It was shown \cite[Theorem 1.1]{KKW} that, if  $\{n_i\}_{i \ge 1}$ is any sequence of elements in $S$ such that ${\cal N}_{n_i} \to \infty$, then 
\begin{align} \label{10luglio2017}
\text{Var} ({\cal L}_n)=c_{n_i} \frac{E_{n_i}}{{\cal N}^2_{n_i}} (1+o(1)), \hspace{1cm} c_{n_i}=\frac{1+ \hat{\mu}^2_{n_i}(4)}{512}, 
\end{align}
where $\hat{\mu}_n(k)\in[-1,1]$ is the Fourier transform of $\mu_n$. The positive real numbers $c_{n_i}$ in the leading constant depend on the angular distribution of $\Lambda_{n_i}$, i.e., in dimension $d=2$, the asymptotic behavior of the variance is {\it non-universal}. 

Also remarkably, the order of magnitude of \eqref{10luglio2017} is much smaller then expected \eqref{10luglio} since the terms of order $E_n/{{\cal N}_n}$ in the asymptotic expression for the nodal length variance cancel perfectly. This effect was called {\it arithmetic Berry cancellation} after the cancellation phenomenon observed by Berry in \cite{MB}.

The limiting distribution of the nodal length was derived in \cite[Theorem 1.1]{MPRW} where it is proved that the normalized nodal length converges to a non-universal, non-Gaussian, limiting distribution, depending on the angular distribution of lattice points. For $\{n_i\}_{i \ge 1} \in S$, such that  ${\cal N}_{n_i} \to \infty$ and $|\hat{\mu}_{n_i}(4)| \to \eta \in  [0,1]$, one has
  \begin{align} \label{10luglio16:55}
 \frac{{\cal L}_n - \mathbb{E}[{\cal L}_n]}{\sqrt{   \text{Var}({\cal L}_n)  }} \stackrel{law}{\to}
 \frac{1}{2 \sqrt{1+ \eta^2}} [2-(1+ \eta) X^2_1 -(1- \eta) X_2^2],
 \end{align}
where $X_1, X_2$ are i.i.d. standard Gaussian. A quantitative version of \eqref{10luglio16:55} was derived in \cite{P&R}. 

As observed in \cite{MPRW, DNPR} the non-central and non-universal behavior of second order fluctuations originates from the {\it chaotic cancellation phenomenon} exploited in \cite{MPRW}: in the Wiener chaos expansion of ${\cal L}_n$ the projection on the second chaos vanishes, and the limiting fluctuations of ${\cal L}_n$ are completely determined by its projection on the fourth Wiener chaos. Should the second projection of ${\cal L}_n$ not disappear in the limit, then the order of the variance would be $E_n/{{\cal N}_n}$.   

The dominance of the fourth-order chaos, and the consequent lower order of the variance, was observed also for the Euler Characteristic of excursion sets of random Laplace eigenfunctions on the $2$-dimensional sphere \cite{CM}. The investigation of the general validity of such asymptotic behavior for other geometric functionals of arithmetic random waves is left for future research.

\subsubsection{Nodal area} The asymptotic behavior of the nodal area variance on $\mathbb{T}^3$ has been recently analysed in \cite{B&M}: the variance of the nodal area has the following precise asymptotic behavior, as $n \to \infty$, $n \not\equiv 0, 4, 7\, (\text{mod 8})$,  
\begin{align} \label{varjr}
\text{Var} ({\cal A}_n)=\frac{n}{{\cal N}^2_n}  \left[ \frac{32}{375}+ O \left( \frac{1}{n^{1/28 -o(1)}} \right) \right].
\end{align}
The condition $n \not\equiv 0, 4, 7\, (\text{mod 8})$ implies ${\cal N}_n \to \infty$ (see Section \ref{10.7.2017}). 

In particular the $3$-dimensional torus exhibits arithmetic Berry cancellation like the $2$-dimensional torus. However, unlike the $2$-dimensional case, the leading order term does not fluctuate, this is due to the {\it equidistribution of lattice points on the sphere} (see Section \ref{10.7}).

\section{Main results and outline of the proof}
Our principal result is the asymptotic distribution, as ${\cal N}_n \to \infty$, of the sequence of normalized nodal areas:  
\begin{theorem} \label{th}
 Let $\textrm{{\Large$\chi$}}$ be a chi-square with $5$ degrees of freedom. As $n \to \infty$, $n \not\equiv 0, 4, 7\, (\rm{mod} \;  8)$,
\begin{align*}
\frac{{\cal A}_n-\mathbb{E}[{\cal A}_n]}{\sqrt{ {\rm Var}({\cal A}_n) }} \stackrel{law}{\to} \frac{1}{\sqrt{5 \cdot 2}} \left(5- \textrm{{\Large$\chi$}} \right).
\end{align*}
\end{theorem}
\vspace{0.3cm}

The first step in the proof of Theorem \ref{th} is the derivation of the Wiener chaos expansion of the nodal area ${\cal A}_n$. \\

In particular in Lemma \ref{caotic} we derive the Wiener chaos expansion of the nodal volume for every dimension $d \ge 2$:  
$${\cal V}_n=\mathbb{E}[{\cal V}_n]+ \sum_{q=1}^{\infty} {\cal V}_n[q],$$
here ${\cal V}_n[q]$, $q=1, 2, \dots$ denotes the orthogonal projection of ${\cal V}_n$ onto the so-called  Wiener chaos of order $q$ (see Section \ref{wiener} below). The proof of Lemma \ref{caotic} (see Appendix \ref{caotic_A}) is a straightforward generalisation of the Wiener chaos expansion of the nodal length performed in \cite[Proposition 3.2]{MPRW} (see also \cite{K&L} for analogous computations involving the length of level curves for Gaussian fields on the Euclidean plane). This $d$-dimensional analogue of  \cite[Proposition 3.2]{MPRW} is shown by performing the $L^2$-expansion of the norm of the $d$-dimensional gradient vector of $f_n$.

\vspace{0.3cm}

As for the nodal length \cite[Section 1.4]{MPRW}, we obtain that both the second-order projection and all odd-order projections vanish. A precise analysis of the fourth-order chaos ${\cal V}_n[4]$ allows as to show the following: 
\begin{align} \label{mar18lug}
{\cal V}_n[4]&=\sqrt{ \frac{n\pi }{d} }  \frac{\Gamma\left(\frac{d+1} 2\right)   }{4 {\cal N}_n \Gamma\left(\frac{d} 2\right)} \Big[ 4 \frac{d-1}{d+2}+ \frac{2}{d+2}W^2(n) - \frac{2 d}{d+2} \sum_{j, k} W^2_{j,k}(n) \\
& \hspace{3.2cm}+ X(n) + 2 \sum_k X_{k,k}(n) - \frac{d}{d+2} \sum_{ j,k} X_{k,k,j,j}(n) +o_{\mathbb{P}}(1)\Big]. \nonumber
\end{align}
where 
\begin{align*}
&\hspace{-3cm}W(n)= \frac{1}{\sqrt{{\cal N}_n }} \sum_{\lambda \in  \Lambda_n} (|a_{\lambda}|^2-1),\\
&\hspace{-3cm}W_{j,k}(n)=\frac{1}{n \sqrt{{\cal N}_n }} \sum_{\lambda \in  \Lambda_n} \lambda_{(k)} \lambda_{(j)} (|a_{\lambda}|^2-1),\\
&\hspace{-3cm}X(n)=\frac{1}{{ \cal N}_n} \sum_{ \lambda_1,\dots , \lambda_4  \in \cal{X}_n(4)} a_{\lambda_1}   a_{\lambda_2}  a_{\lambda_3}  a_{\lambda_4}, \\
& \hspace{-3cm}X_{k,k}(n)=\frac{1}{n { \cal N}_n} \sum_{\lambda_1,\dots , \lambda_4  \in \cal{X}_n(4)} \lambda_{1,(k)}  \lambda_{2,(k)}  a_{\lambda_1}   a_{\lambda_2}  a_{\lambda_3}  a_{\lambda_4}, \\
&\hspace{-3cm}X_{k,k,j,j}(n)= \frac{1}{n^2{ \cal N}_n} \sum_{\lambda_1,\dots , \lambda_4  \in \cal{X}_n(4)} \lambda_{1,(k)}  \lambda_{2,(k)} \lambda_{3,(j)}  \lambda_{4,(j)}  a_{\lambda_1}   a_{\lambda_2}  a_{\lambda_3}  a_{\lambda_4},  
\end{align*}
$\lambda_{(k)}$ denotes the $k$-th component of $\lambda$, and $\cal{X}_n(4)$ is the set of $d$-dimensional lattice point, non-degenerate, $4$-correlations defined in Section \ref{spectral} below (see also \cite[Section 1.4]{B&M}). \\

We remark that in dimension $d=2$ 
$$| \cal{X}_n(4)|=0,$$
for all $n \in S$, which may be seen by noting that two circles intersect in at most two points (Zygmund's trick), so the asymptotic behaviour of the nodal length studied in \cite{MPRW} comes form a precise analysis of the asymptotic behavior of the first three terms in \eqref{mar18lug}:
\begin{align*}
{\cal L}_n[4]&=\pi \sqrt{ \frac{n }{512} }  \frac{1 }{ {\cal N}_n } \Big[2+  W^2(n) - 2  \sum_{j, k} W^2_{j,k}(n) +o_{\mathbb{P}}(1)\Big]. \nonumber
\end{align*}

Our proof of Theorem \ref{th} relies on recent results by Benatar and Maffiucci \cite[Section 1.4]{B&M} (see Lemma \ref{JB} below) showing that, in dimension $d=3$, the tuples that cancel pairwise dominate the non-degenerate tuples in the high frequency limit. In particular in Lemma \ref{X&Xjk} we prove that as  $n \to \infty$, $n \not\equiv 0, 4, 7 \, (\rm{mod} \, 8)$
\begin{align*}
X(n), X_{k,k}(n), X_{k,k,j,j} (n) \stackrel{L^2}{\to} 0.
\end{align*}
This implies that 
\begin{align*}
{\cal A}_n[4]
&=\frac{\sqrt n}{5  \sqrt 3 \, {\cal N}_n} \Big[ 4+ W^2(n) - 3 \sum_{j, k} W^2_{j,k}(n)  + o_{\mathbb{P}}(1)\Big].
\end{align*}

To prove Theorem \ref{th} we first show that the normalized fourth-order projection of the nodal area converges to a universal, non-Gaussian, distribution: as  ${\cal N}_n \to \infty$ 
 \begin{align*}
\frac{{\cal A}_n[4]}{\sqrt{ {\rm Var}({\cal A}_n[4]) }} \stackrel{law}{\to} \frac{1}{\sqrt{5 \cdot 2}} \left(5-   \textrm{{\Large$\chi$}} \right),
\end{align*}
where $\textrm{{\Large$\chi$}}$ is a chi-square with $5$ degrees of freedom. The derivation of such limiting distribution requires a precise analysis of the asymptotic behavior of the covariance matrix of the $W_{j,k}(n)$, i.e. the asymptotic behaviour of the following quantities  
\begin{align*}
\frac{1}{ n^2 {\cal N}_n}{\sum_{\lambda \in \Lambda_n}} \lambda^4_{(k)}=\frac{1}{ 5}+O\left( \frac{1}{n^{1/28 -o(1)}}\right), \hspace{0.9cm}
\frac{1}{ n^2 {\cal N}_n}{\sum_{\lambda \in \Lambda_n}} \lambda^2_{(k)} \lambda^2_{(j)}=\frac{1}{3 \cdot 5}+O\left( \frac{1}{n^{1/28 -o(1)}}\right).
\end{align*} 
This is obtained in Lemma \ref{lambda_jk} as an application of the equidistribution of lattice points on spheres proved by Duke \cite{D, D3} and Golubeva and Fomenko \cite{G&F} (see \cite[Lemma 8]{P&S} or Lemma \ref{equid} below). 

The last step in the proof of Theorem \ref{th} is the observation that, as in \cite{MPRW}, the nodal area is dominated by its fourth-order chaos component: using the asymptotic behaviour of the variance in \eqref{varjr} it is easy to check the equivalence
\begin{align*}
{\rm Var}({\cal A}_n) \sim {\rm Var}({\cal A}_n[4]),
\end{align*}
as ${\cal N}_n \to \infty$. Theorem \ref{th} follows since different chaotic projections are orthogonal in $L^2$.\\

In dimension $d\ge 5$ the set of non-degenerate tuples ${\cal X}_n(4)$ is much larger than $\cal{D}_n(4)$ \cite{B&M}, as opposed to what happens in dimensions $2$ and $3$. This implies that the derivation of the asymptotic behavior of variance and limiting distribution of the nodal volume requires a precise analysis of the structure of the non-degenerate tuples ${\cal X}_n(4)$ which seems to be very technically demanding.

\subsection{Notation} For functions $f$ and $g$ we will use Landau's asymptotic notation $$f=O(g)$$ to denote that $f \le C g$ for some constant $C$. With $\stackrel{law}{\to}$ we denote weak convergence
of probability measures and we use the symbol $o_{\mathbb{P}}(1)$ to denote a sequence of random variables converging to zero in probability.  

We will use $\lambda, \lambda_1, \lambda_2 \dots$ and in general $\lambda_i$, $i=1,2,\dots$ to denote elements of $\Lambda_n$, while  $\lambda_{(k)}$ and $\lambda_{i,(k)}$ with $k=1, \dots, d$, will denote the $k$th component of the vectors $\lambda$ and $\lambda_i \in \Lambda_n$ respectively. The indices $j,k$ always run from $1$ to $d$.

For $k=1,\dots d$, we denote with $\partial_k f_n(x)$ the derivative of $ f_n(x)$ with respect to $x_k$
\begin{align*}
\partial_k f_n(x)
= \frac{2 \pi i}{\sqrt{{\cal N}_n}}  \sum_{\lambda \in \Lambda_n} a_{\lambda} \lambda_{(k)} e(\langle \lambda,x\rangle),
\end{align*}
in view of  \cite[Lemma 2.3]{R&W}, see formula \eqref{RW} below, the random field $\partial_k f_n$ has variance 
\begin{align*}
\text{Var}( \partial_k f_n(x) )&= \frac{2^2 \pi^2 (-1)}{{\cal N}_n}  \sum_{\lambda_1, \lambda_2  \in \Lambda_n} \mathbb{E}[a_{\lambda_1} a_{\lambda_2} ] \, \lambda_{1,(k)} \, \lambda_{2,(k)} \, e(\langle \lambda_1 ,x\rangle) \, e(\langle \lambda_2 ,x\rangle)\\
&= \frac{2^2 \pi^2 }{{\cal N}_n}  \sum_{\lambda \in \Lambda_n}  \lambda^2_{(k)}=\frac{2^2 \pi^2 n }{d} ,
\end{align*}
we introduce then the normalized derivative $f_{n,k}(x)$ defined by 
\begin{align} \label{norm_d}
f_{n,k}(x)= \frac{\partial_k f_n(x)}{2 \pi   \sqrt{ \frac{n }{d} }  }= i \sqrt{\frac{d}{n {\cal N}_n}}  \sum_{\lambda \in  \Lambda_n} \lambda_k a_{\lambda} e(\langle \lambda, x \rangle).
\end{align}
Note that $f_{n,k}(x)$ is real-valued since $f^2_{n,k}(x)=f_{n,k}(x) \overline{f_{n,k}}(x)$. We note that conditions \ref{i}) and \ref{iii}) in \eqref{fn} immediately imply that  
$${\rm Var } (a^2_{\lambda})=\mathbb{E}[a^2_{\lambda}]=\mathbb{E}[{\rm Re}(a_{\lambda})] - \mathbb{E}[{\rm Im}(a_{\lambda})]=0,$$
and that $2 |a_{\lambda}|^2$ has a chi-squared distribution with $2$ degrees of freedom:   
\begin{align*}
\mathbb{E}[ |a_{\lambda}|^2]=
1, \hspace{1cm}  \mathbb{E}[ (|a_{\lambda}|^2-1)^2]={\rm Var } (|a_{\lambda}|^2) =
1, \hspace{1cm} \mathbb{E}[|a_{\lambda}|^4]=2.
\end{align*} 
We also define
\begin{align*}
&R(n)=\frac{1}{{\cal N}_n} \sum_{\lambda \in  \Lambda_n} |a_{\lambda}|^4, \hspace{2.3cm} R_{k,j}(n)=\frac{1}{n^2 {\cal N}_n} \sum_{\lambda \in  \Lambda_n} \lambda_{(k)}^2  \lambda_{(j)}^2  |a_{\lambda}|^4. 
\end{align*}

\subsection{Acknowledgements} The research leading to these results has received funding from the European Research Council under the European Union's Seventh Framework Programme (FP7/2007-
2013)/ERC grant agreements no. 335141. The author wish to thank Igor Wigman for drawing her attention to this problem and for stimulating discussions. The author is grateful to Jacques Benatar, Domenico Marinucci and Ze\'ev Rudnick for insightful remarks.

\section{Lattice points on spheres: spectral correlations and equidistribution}

\subsection{Properties of the frequency set} \label{10.7.2017}

The dimension ${\cal N}_n= \text{dim} \, {\cal E}_n$ is the number of ways of expressing $n$ as a sum of $d$ integer squares. In dimension $d \ge 5$, $\cal{N}_n$ grows roughly as $n^{d / 2 -1}$ as $n \to \infty$ \cite[Theorem 20.9]{I&K}. For $d \le 4$ the dimension of the eigenspace need not grow with $n$ and the behavior of $\cal{N}_n$ is more erratic. 

If $d=2$, ${\cal N}_n$ is given in terms of the prime decomposition of $n$ as follows \cite[Section 16.9]{H&W}: for 
$$n=2^{\alpha} \prod_j p_j^{\beta_j} \prod_k q_k^{2 \gamma_k},$$ where $p_j, q_k$ are odd primes $p_j \equiv 1\; \text{mod}\, 4$ and $q_k \equiv 3\; \text{mod}\, 4$ and $\alpha, \beta_j, \gamma_k$ are positive integers, then $\cal{N}_n=4 \prod_j(\beta_j+1)$ or otherwise $n$ is not a sum of two squares and ${\cal N}_n=0$. ${\cal N}_n$ is subject to large and erratic fluctuations; it grows {\it on average} \cite{La}, over integers which are sums of two squares, as  $\text{const} \cdot \sqrt{\log n}$, but could be as small as $8$ for an infinite sequence of prime numbers $p \equiv 1\; \text{mod}\, 4$, or as large as a power of $\log n$.   

In dimension $d=3$, by a classical result of Legendre and Gauss, $n$ is a sum of three squares if and only if $n \not \equiv 4^a (8 b+7)$. The behavior of ${\cal N}_n$ is very subtle \cite[Section 1]{BR&S} and it was shown in the 1930's  that ${\cal N}_n$ goes to infinity with $n$, assuming that $n$ is square-free (if $n = 4^a$ then there are only six solutions). It is known that ${\cal N}_n \ll n^{1/2+o(1)}$. If there are primitive lattice points
, which happens if and only if $n \not\equiv 0, 4, 7 \; ({\rm mod} \;8 )$, then there is a lower bound ${\cal N}_n \gg n^{1/2-o(1)}$.\\ 
 
The frequency set $\Lambda_n$ is invariant under the group $W_d$ of signed permutations, consisting of coordinate permutation and sign-change of any coordinate. In particular, $\Lambda_n$ is symmetric under $\lambda \to - \lambda$, and since $0 \notin \Lambda_n$, ${\cal N}_n$ is even. Using invariance under $W_d$ in  \cite[Lemma 2.3]{R&W} it is proved the following lemma: 

\begin{lemma} \label{grouplemma}
For any subset ${\cal O} \subset \Lambda_n$ which is invariant under the group $W_d$, we have
\begin{equation} \label{RW}
\sum_{\lambda \in {\cal O} } \lambda_{(j)} \lambda_{(k)}=|{\cal O}| \frac{n }{d} \delta_{j,k}.
\end{equation}
\end{lemma}
\noindent We note that using the invariance of $\Lambda_n$ under the group $W_d$, we also immediately obtain that
\begin{align} \label{inv}
{\sum_{\lambda \in \Lambda_n}} \prod_{i=1}^d \lambda^{\alpha_i}_{(i)}  = 0,
\end{align}
if at least one of the exponents  $\alpha_i$ is odd.

\subsection{Equidistribution of lattice points on spheres} \label{10.7}
The classical Linnik problem \cite{D2} about the distribution of lattice points on a sphere was first introduced and discussed by Linnik in \cite{L}.

 In his book Linnik asked whether the points $\Lambda_n/ \sqrt n$, obtained by projecting the set $\Lambda_n=\{\lambda \in \mathbb{Z}^3: ||\lambda||^2=n \}$ to the unit sphere ${\cal S}^2$, become equidistributed with respect to the (normalized) Lebesgue measure $d \sigma$ on ${\cal S}^2$ as $n \to \infty$, subject to the condition that $n \not\equiv 0, 4, 7 \, ( \text{mod} \, 8)$. 
 
 Linnik was able to solve the problem using his {\it ergodic method} and assuming the Generalised Riemann
Hypothesis. The Linnik problem was solved unconditionally by Duke \cite{D, D3} and Golubeva and Fomenko \cite{G&F}, following a breakthrough by Iwaniec \cite{I} on modular forms. 

As a consequence we may approximate a summation over the lattice point set by an integral over the unit sphere as follows:  

\begin{lemma}[{\cite[Lemma 8]{P&S}}] \label{equid}
Let $g \in C^\infty({\cal S}^2)$, for every $n \not \equiv 0, 4, 7 \, ( \rm{mod} \, 8)$, we have 
\begin{equation} \label{equid_int}
\frac{1}{{\cal N}_n} \sum_{\lambda \in \Lambda_n} g\left( \frac{\lambda}{||\lambda||} \right)= \int_{{\cal S}^2} g(u) \; d \sigma(u)+O\left( \frac{1}{n^{1/28-o(1)}} \right).
\end{equation}
\end{lemma}
\vspace{0.3cm}

\noindent From Lemma \ref{equid} it immediately follows that 
\begin{lemma} \label{lambda_jk} For $k, j=1, 2, 3$, $k \ne j$, we have 
\begin{align} \label{4}
\frac{1}{ n^2 {\cal N}_n}{\sum_{\lambda \in \Lambda_n}} \lambda^4_{(k)}=\frac{1}{ 5}+O\left( \frac{1}{n^{1/28 -o(1)}}\right), 
\end{align} 
\begin{align} \label{22}
\frac{1}{ n^2 {\cal N}_n}{\sum_{\lambda \in \Lambda_n}} \lambda^2_{(k)} \lambda^2_{(j)}=\frac{1}{3 \cdot 5}+O\left( \frac{1}{n^{1/28 -o(1)}}\right). 
\end{align} 
\end{lemma} 
\begin{proof}
By invariance under the group $W_3$, the integrand, at the right-hand side of \eqref{equid_int}, is a function of only one angle. We apply Lemma \ref{equid} and we obtain  
\begin{align*}
\frac{1}{ n^2 {\cal N}_n}{\sum_{\lambda \in \Lambda_n}} \lambda^4_{(k)}&=\frac{1}{ {\cal N}_n}{\sum_{\lambda \in \Lambda_n}} \left(\frac{\lambda_{(k)}}{||\lambda||}\right)^4\\
& = \frac{1}{4 \pi} \int_0^{\pi}  d \phi_1  \int_{0}^{2 \pi} d \phi_2  \cos^4 \phi_1 \sin \phi_1   +  O\left( \frac{1}{n^{1/28-o(1)}} \right) \\
&=  \frac{2 \pi}{4 \pi}  \frac{2}{5}+  O\left( \frac{1}{n^{1/28-o(1)}} \right). 
\end{align*} 
The proof of \eqref{22} is similar.  
\end{proof}
Malyshev \cite{M} and Pommerenke \cite{P} have established the analogue of Lemma \ref{equid} for integral positive quadratic forms in more than three variables. 
\subsection{Spectral correlations} \label{spectral}

For $\ell \ge 2$, we denote by ${\cal C}={\cal C}_n(\ell)$ the set of $d$-dimensional $\ell$-{\it correlations}: 
$${\cal C}_n(\ell)=\{ (\lambda_1, \dots, \lambda_{\ell}) \in ({\cal E}_n)^{\ell} : \, \sum_{i=1}^{\ell} \lambda_i=0 \}.$$
The set of {\it non-degenerate} $\ell$-correlations ${\cal X}={\cal X}_n(\ell)$ is the subset of ${\cal C}$ defined by 
$${\cal X}_n(\ell)=\{ (\lambda_1, \dots, \lambda_{\ell}) \in {\cal C}_n(\ell): \; \forall {\cal H} \subset  \{1, \dots, \ell \}, \sum_{i \in {\cal H}}  \lambda_i \ne 0  \},$$
and we denote by ${\cal D}={\cal C} \setminus {\cal X}$ the set of {\it degenerate} correlations. For $d>2$ a summation over ${\cal C}(4)$ may be treated by separating it as follows

\begin{align} \label{M}
\sum_{{\cal C}(4)} &= \sum_{\substack {\lambda_1=-\lambda_2 \\ \lambda_3 = - \lambda_4 }} +  \sum_{ \substack{\lambda_1=-\lambda_3 \\ \lambda_2= - \lambda_4 }} +  \sum_{ \substack {\lambda_1=-\lambda_4 \\ \lambda_2 = - \lambda_3}} \\
&\;\;\;- \sum_{\lambda_1=-\lambda_2=\lambda_3=-\lambda_4} -  \sum_{\lambda_1=\lambda_2=-\lambda_3=-\lambda_4 } -  \sum_{\lambda_1=- \lambda_2=-\lambda_3=\lambda_4 }     \nonumber  \\
&\;\;\;+ \sum_{{\cal X}(4)}.   \nonumber 
\end{align}

\noindent Note that if $d=2$ the set ${\cal X}(4)$ is empty. The next lemma deals with the $3$-dimensional setting and provides an estimate for the number of non-degenerate correlations: 
\begin{lemma}[{\cite[Theorem 1.5]{B&M}}] \label{JB}
Letting $n \to \infty$, one has the estimate 
\begin{align*}
|{\cal X}_n(4)| \ll {\cal N_n}^{7/4+o(1)}.
\end{align*}
\end{lemma}

\section{Chaotic expansion of ${\cal V}_n$} 

\subsection{Wiener chaos expansion} \label{wiener}

The celebrated Wiener chaos expansion \cite{W} concerns the representation of square
integrable random variables in terms of an infinite orthogonal sum. In this section we recall briefly some basic facts on Wiener chaotic expansion for non-linear functionals of Gaussian fields, we refer to \cite{N&P} for an exhaustive discussion. 

Denote by $\{H_k\}_{k \ge 0}$ the Hermite polynomials on $\mathbb{R}$, defined as follows 
\begin{align} \label{hermite}
H_0 =1, \hspace{0.5cm} H_k(t)=(-1)^k \gamma^{-1}(t) \frac{d^k}{d t^k} \gamma(t), \;\; k \ge 1,
\end{align} 
where $\gamma(t)=e^{-t^2/2}/\sqrt{2 \pi}$ is the standard Gaussian density on the real line; $\mathbb{H}= \{ H_k / \sqrt{k!}: \; k \ge 0\}$ is a complete orthogonal system in 
$$L^2(\gamma)= L^2(\mathbb{R}, {\cal B} (\mathbb{R}), \gamma(t) d t).$$

The random eigenfunctions $f_n$ defined in \eqref{fn} are a byproduct of the family of complex-valued, Gaussian random variables $\{a_{\lambda}\}$, defined on some probability space $(\Omega, {\cal F},\mathbb{P})$.  Following the discussion in \cite{MPRW} we define the space ${\bf A}$ to be the closure in $L^2(\mathbb{P})$ generated by all real, finite, linear combinations of random variables of the form $z a_{\lambda}+ \overline z a_{-\lambda}$, $z \in \mathbb{C}$; the space ${\bf A}$ is a real, centred, Gaussian Hilbert subspace of $L^2(\mathbb{P})$.  \\

For each integer $q \ge 0$, the $q$-th {\it Wiener chaos} ${\cal H}_q$ associated with ${\bf A}$ is the closed linear subspace of $L^2(\mathbb{P})$ generated by all real, finite, linear combinations of random variables of the form 
$$H_{q_1}(a_1) \cdot H_{q_2}(a_2) \cdots H_{q_k}(a_k)$$
for $k \ge 1$, where the integers $q_1, q_2, \dots, q_k \ge 0$ satisfy $q_1+q_2+ \cdots+q_k=q$ and $(a_1,a_2,\dots,a_k)$ is a real, standard, Gaussian vector extracted from ${\bf A}$. In particular ${\cal H}_0=\mathbb{R}$.\\

As well-known Wiener chaoses $\{{\cal H}_q, \, q=0,1,2, \dots\}$ are orthogonal \cite[Theorem 2.2.4]{N&P}, i.e., ${\cal H_q} \perp {\cal H_p}$ for $p \ne q$ (the orthogonality holds in the sense of $L^2(\mathbb{P})$) and the following decomposition holds: every real-valued function $F \in {\bf A}$ admits a unique expansion of the type 
\begin{align*} F = \sum_{q=0}^{\infty} F[q],
\end{align*}
where the projections $F[q] \in {\cal H}_q$ for every $q=0,1,2,\dots$, and the series converges in $L^2(\mathbb{P})$. Note that $F[0]=\mathbb{E}[F]$.

\subsection{Chaotic expansion of ${\cal V}_n$} 

In this section we derive the explicit form for the projections in the chaos decomposition of the nodal volume ${\cal V}_n$. For the proof of  Lemma \ref{caotic} we follow closely \cite[Section 3]{MPRW} where the chaotic expansion of the nodal length of arithmetic random waves is derived.

\subsubsection{Approximating the nodal volume}

Let $1_{[-\varepsilon, \varepsilon]}$ be the indicator function of the interval $[-\varepsilon, \varepsilon]$ and $||\cdot||$ the standard Euclidean norm in $\mathbb{R}^d$. We define for $\varepsilon>0$ 
\begin{align*} 
{\cal V}^{\varepsilon}_n= \frac{1}{2 \varepsilon} \int_{\mathbb{T}^d} 1_{[-\varepsilon, \varepsilon]} (f_n(x)) \; || \nabla f_n(x)|| d x. 
\end{align*}
 It was shown in \cite[Lemma 3.1]{R&W} that a.s.  
\begin{align} \label{uuno} 
{\cal V}_n=\lim_{\varepsilon \to \infty} {\cal V}^{\varepsilon}_n
\end{align}
and that \cite[Lemma 3.2]{R&W} ${\cal V}^{\varepsilon}_n$ is uniformly bounded, that is  
\begin{align} \label{ddue}
{\cal V}^{\varepsilon}_n \le 6 \, d \sqrt{E_n}.
\end{align} 

The Dominated Convergence Theorem and the uniform bound \eqref{ddue} imply that the convergence in \eqref{uuno} is in $L^2(\mathbb{P})$, i.e., for every $n \in S$ 
\begin{align} \label{ttre}
\lim_{\varepsilon \to 0} \mathbb{E}[ | {\cal V}^{\varepsilon}_n - {\cal V}_n  |^2 ]=0.
\end{align}

\subsubsection{Chaos expansion}
In view of \eqref{ttre}, we first compute the chaotic expansion of ${\cal V}^{\varepsilon}_n$, then the expansion of ${\cal V}_n$ in Lemma \ref{caotic} follows by letting $\varepsilon \to 0$. In Lemma  \ref{caotic} we prove that all odd-order chaotic components in the Wiener chaos expansion of ${\cal V}_n$ vanish and we derive an explicit expression for the even-order chaotic components:

\begin{lemma} \label{caotic}  The Wiener chaos expansion of ${\cal V}_n$ is
${\cal V}_n=\mathbb{E}[{\cal V}_n]+ \sum_{q=1}^{\infty} {\cal V}_n[q]$
 in $L^2(\mathbb{P})$; where ${\cal V}_n[2 q+1]=0$ for $q \ge 1$, and for $q > 1$
\begin{align*} 
{\cal V}_n[2 q]&=2 \pi   \sqrt{ \frac{n }{d} }    \sum_{p=0}^{q}    \frac{\beta_{2 q-2p}}{(2q-2p)!}   \hspace{-0.2cm} \sum_{\stackrel{s \in {\mathbb N}^d}{s_1+ \cdots+s_d=p}} a(2 s)  \int_{\mathbb{T}^d} H_{2q-2p}(f_n(x))  \prod_{j=1}^d H_{2 s_j} (f_{n,j}(x)) d x;
\end{align*}
with
\begin{align} 
& \beta_{2 m-2p}=\frac{1}{\sqrt{2 \pi}} H_{2 m-2p }(0),  \label{18:30} \\
&a(2q)= \sum_{i=0}^{\infty} \frac{1}{i!\; 2^i }  \frac{ \sqrt 2 \Gamma(\frac d 2 +i+ \frac 1 2) }{ \Gamma(\frac{d}{2}+i) } \sum_{j_1+\cdots+j_d=i} \binom{i}{j_1, \dots , j_d}  \frac{(-1)^{q_1-j_1+\cdots+q_d-j_d}}{(q_1-j_1)! \cdots (q_d-j_d)! 2^{q_1-j_1+\cdots+q_d-j_d}}. \label{18:48}
\end{align}
\end{lemma}
The proof of Lemma \ref{caotic} is postponed to Appendix \ref{caotic_A}.

\subsubsection{Second and fourth order chaos}

For $j,k =1, \dots, d$, we denote by $s(j)$ the vector in $\mathbb{R}^d$ with $j$-th component equal to $1$ and all the other components equal to $0$, and we denote by  $s(j,k)$ the vector in  $\mathbb{R}^d$ with $j$-th and $k$-th components equal to $1$ and all the other components equal to $0$. To evaluate the second and fourth order chaos we need the following lemma: 
\begin{lemma} \label{aeb}
\begin{align*}
\beta_0=\frac{1}{\sqrt{2 \pi}},\;\;\;\; \beta_2=-\frac{1}{\sqrt{2 \pi}},\;\;\;\; \beta_4=\frac{3}{\sqrt{2 \pi}},
\end{align*}
for $k=1, \dots, d$, 
\begin{align*}
a(0)=\sqrt 2 \frac{\Gamma(\frac{d+1} 2 )}{\Gamma(\frac d 2)}, \;\;\; a(2 s(k))=\frac{1}{2 \sqrt 2 } \frac{\Gamma(\frac{d+1}2)}{\Gamma(\frac{d+2}2)}, \;\;\; a(4 s(k))=-\frac{1}{2^4 \sqrt 2 } \frac{\Gamma(\frac{d+1}2)}{\Gamma(\frac{d+4}2)},  
\end{align*}
and for $j \ne k$, $j,k =1, \dots, d$, 
\begin{align*}
a(2 s(j,k))=-\frac{1}{2^3 \sqrt 2 } \frac{\Gamma(\frac{d+1} 2)}{\Gamma(\frac{d+4}2)}.
\end{align*}
\end{lemma}
The proof of Lemma \ref{aeb} is postponed to Appendix \ref{2&4_ch_A}. Using Lemma \ref{aeb} we easily see that we can rewrite the second-order chaos and the fourth-order chaos as follows 
\begin{align}
{\cal V}_n[2] &= \sqrt{\frac{n \pi}{d}} \Gamma\left(\frac{d+1}2\right) \Big[ -  \frac{1}{\Gamma(\frac d 2)} \int_{\mathbb{T}^d} H_2 (f_n(x)) d x + \frac{1}{2 } \frac{1}{\Gamma(\frac{d+2} 2)} \sum_{k=1}^d     \int_{\mathbb{T}^d} H_2(f_{n,k}(x)) dx \Big], \nonumber \\
{\cal V}_n[4]&=\sqrt{ \frac{n\pi }{d} }   \Gamma\left(\frac{d+1} 2\right)  \Big[ \frac{1}{2^2  \Gamma(\frac{d}{2})} \int_{\mathbb{T}^d} H_{4}(f_n(x)) d x  \label{17lug} \\
&\;\;  \hspace{3.2cm}-   \frac{1}{ 2^2 \Gamma(\frac{d+2}2)}
\sum_{k=1}^d \int_{\mathbb{T}^d} H_2(f_n(x)) H_2(f_{n,k}(x)) d x  \nonumber \\
&\;\; \hspace{3.2cm} -    \frac{1}{ 2^4 \Gamma(\frac{d+4}2)}  \sum_{k=1}^d \int_{\mathbb{T}^d} H_4(f_{n,k}(x)) d x  \nonumber \\
&\;\; \hspace{3.2cm} -  \frac{1}{ 2^3 \Gamma(\frac{d+4}2)} \sum_{j < k}  \int_{\mathbb{T}^d} H_2(f_{n,j}(x)) H_2(f_{n,k}(x)) d x\Big].  \nonumber 
\end{align}

We first prove that the second-order chaotic projection vanishes, i.e.  ${\cal V}_n[2]=0$. This is an immediate consequence of the following lemma
\begin{lemma} \label{14:20} 
\begin{align}
& \int_{\mathbb{T}^d}  H_{2}(f_n(x))  d x= \frac{1}{\sqrt{{\cal N}_n}} W(n), \label{11st}\\
&\sum_{k=1}^d  \int_{\mathbb{T}^d}  H_{2}(f_{n,k}(x))  d x= \frac{d}{\sqrt{{\cal N}_n}} W(n). \label{22nd} 
\end{align}
\end{lemma}
The proof of Lemma \ref{14:20} is postponed to Appendix \ref{2&4_ch_A}. The precise analysis of the fourth-order chaotic component also relies on the following representation lemma
\begin{lemma} \label{15:10}
\begin{align*}
&\int_{\mathbb{T}^d}  H_{4}(f_n(x))dx= \frac{3}{{\cal N}_n} W^2(n)  -  \frac{3}{{\cal N}_n}  R(n)  +   \frac{1}{{\cal N}_n} X(n),\\
&\sum_{k=1}^d \int_{\mathbb{T}^d}   H_{2}(f_n(x))  H_{2}(f_{n,k}(x))  dx= \frac{d}{{\cal N}_n} W^2(n)    - \frac{d}{{\cal N}_n}  R(n)- \frac{d}{ {\cal N}_n} \sum_{k=1}^d X_{k,k}(n),\\
& \int_{\mathbb{T}^d} H_4(f_{n,k}(x)) d x=  \frac{3 d^2}{{\cal N}_n}  W^2_{k,k} (n)- \frac{3 d^2}{ {\cal N}_n}  R_{k,k}(n)  +\frac{d^2}{ {\cal N}_n} X_{k,k,k,k}(n), \\
& \sum_{j \ne k} \int_{\mathbb{T}^d} H_2(f_{n,j}(x) )  H_2(f_{n,k}(x) ) d x =   \frac{d^2}{ {\cal N}_n}  W^2(n) - \frac{d^2}{ {\cal N}_n} \sum_{k=1}^d W^2_{k,k}(n)  +   \frac{2 d^2}{{\cal N}_n} \sum_{j \ne k}  W^2_{j,k} (n) \\
& \hspace{6cm} -    \frac{ 3 d^2}{ {\cal N}_n} \sum_{j \ne k} R_{k,j}(n) +  \frac{d^2}{ {\cal N}_n} \sum_{j \ne k} X_{k,k,j,j}(n).
\end{align*}
\end{lemma}
Lemma \ref{15:10} is proved in Appendix \ref{2&4_ch_A}. The key tool in the proof of Lemma \ref{15:10} is Lemma \ref{inc&esc} where we evaluate summations over ${\cal C}(4)$ using the structure in \eqref{M}: 
 \begin{lemma} \label{inc&esc}
\begin{align}
&\sum_{{\cal C}(4)} a_{\lambda_1} a_{\lambda_2} a_{\lambda_3} a_{\lambda_4}= 3 \sum_{\lambda_1, \lambda_2} |a_{\lambda_1}|^2  |a_{\lambda_2}|^2 - 3 \sum_{\lambda} |a_{\lambda}|^4 + \sum_{{\cal X}(4)} a_{\lambda_1} a_{\lambda_2} a_{\lambda_3} a_{\lambda_4}, \label{1st}\\
& \sum_{ {\cal C}(4)} \lambda_{1,(k)}\, \lambda_{2,(k)} \, a_{\lambda_1} a_{\lambda_2} a_{\lambda_3} a_{\lambda_4}  = - \sum_{\lambda_1, \lambda_2}  \lambda_{1,(k)}^2 |a_{\lambda_1}|^2  |a_{\lambda_2}|^2  +  \sum_{\lambda} \lambda^2_{(k)}  |a_{\lambda}|^4 \label{2nd} \\
&\hspace{5cm} + \sum_{{\cal X}(4)} \lambda_{1,(k)}\, \lambda_{2,(k)} \, a_{\lambda_1} a_{\lambda_2} a_{\lambda_3} a_{\lambda_4},   \nonumber \\
&\sum_{{\cal C}(4)} \lambda_{1,(k)} \lambda_{2,(k)} \lambda_{3,(k)} \lambda_{4,(k)} a_{\lambda_1} a_{\lambda_2} a_{\lambda_3} a_{\lambda_4}= 3 \sum_{\lambda_1, \lambda_2} \lambda^2_{1,(k)}   \lambda^2_{2,(k)} |a_{\lambda_1}|^2  |a_{\lambda_2}|^2 - 3 \sum_{\lambda} \lambda^4_{(k)} |a_{\lambda}|^4 \label{3rd} \\
&\hspace{6.5cm}+ \sum_{{\cal X}(4)}    \lambda_{1,(k)} \lambda_{2,(k)} \lambda_{3,(k)} \lambda_{4,(k)} a_{\lambda_1} a_{\lambda_2} a_{\lambda_3} a_{\lambda_4},  \nonumber\\
& \sum_{{\cal C}(4)}   \lambda_{1,(k)}  \lambda_{2,(k)} \lambda_{3,(j)} \lambda_{4,(j)} a_{\lambda_1} a_{\lambda_2} a_{\lambda_3} a_{\lambda_4} = \sum_{\lambda_1, \lambda_2} \lambda^2_{1,(k)}  \lambda^2_{2,(j)}  |a_{\lambda_1}|^2 |a_{\lambda_2}|^2 \label{4th}\\
& \hspace{6.5cm} + 2 \left[ \sum_{\lambda} \lambda_{(j)} \lambda_{(k)} (|a_{\lambda}|^2-1) \right]^2 - 3 \sum_{\lambda} \lambda^2_{(k)} \lambda^2_{(j)} |a_{\lambda}|^4  \nonumber \\
 &\hspace{6.5cm}+ \sum_{{\cal X}(4)}   \lambda_{1,(k)}  \lambda_{2,(k)} \lambda_{3,(j)} \lambda_{4,(j)} a_{\lambda_1} a_{\lambda_2} a_{\lambda_3} a_{\lambda_4}. \nonumber
\end{align}
\end{lemma}
\noindent The proof of Lemma \ref{inc&esc} is in Appendix \ref{2&4_ch_A}. Combining \eqref{17lug} and Lemma \ref{15:10} leads to the following representation of the fourth-order chaotic component of the nodal volume ${\cal V}_n[4]$, 
\begin{align} \label{tutto}
{\cal V}_n[4]&=\sqrt{ \frac{n\pi }{d} }  \frac{\Gamma\left(\frac{d+1} 2\right)   }{4 {\cal N}_n \Gamma\left(\frac{d} 2\right)} \Big[  \frac{2}{d+2}W^2(n) - \frac{2 d}{d+2} \sum_{j, k} W^2_{j,k}(n)  -R(n)+ \frac{3 d}{d+2} \sum_{j, k} R_{k,j}(n) \\
& \hspace{3.2cm}+ X(n) + 2 \sum_k X_{k,k}(n) - \frac{d}{d+2} \sum_{ j,k} X_{k,k,j,j}(n)\Big]. \nonumber
\end{align}

\section{Proof of Theorem \ref{th}} \label{theproof}

\subsection{Asymptotic behavior of the fourth order chaos in dimension $d=3$}

The aim of this section is the analysis of the asymptotic behavior, as $n \to \infty$, $n \not\equiv 0, 4, 7\, (\rm{mod} \;  8)$, of the sequence
\begin{align*}
\frac{{\cal A}_{n}[4]}{\sqrt{\text{Var}( {\cal A}_{n}[4])}}.
\end{align*}
From \eqref{tutto} we know that
\begin{align*}
{\cal A}_n[4]
&=\frac{\sqrt n}{5 \sqrt 3 \, {\cal N}_n} \Big[  W^2(n) - 3 \sum_{j, k} W^2_{j,k}(n)  - \frac{5}{2}R(n)+ \frac{3^2}{2 } \sum_{j, k} R_{k,j}(n) \\
& \hspace{2.3cm}+\frac{5}{2} X(n) + 5 \sum_k X_{k,k}(n) - \frac{3}{2 } \sum_{ j,k} X_{k,k,j,j}(n)\Big] .
\end{align*}
In Lemma \ref{R&Rjk} we prove that 
\begin{align*}
 - \frac{5}{2}R(n)+ \frac{3^2}{2 } \sum_{j, k} R_{k,j}(n)= 4 + o_{\mathbb{P}}(1)
\end{align*}  
and in Lemma \ref{X&Xjk} we obtain that  
\begin{align*}
\frac{5}{2} X(n) + 5 \sum_k X_{k,k}(n) - \frac{3}{2 } \sum_{ j,k} X_{k,k,j,j}(n)=o_{\mathbb{P}}(1). 
\end{align*}

\begin{lemma} \label{R&Rjk} As $n \to \infty$, $n \not\equiv 0, 4, 7 \, (\rm{mod} \, 8)$
\begin{align}
R(n)& \stackrel{\mathbb{P}}{\to} 2,  \label{lln_1}\\
R_{k,j}(n)& \stackrel{\mathbb{P}}{\to} \begin{cases}  \frac 2 5 , & {\rm if\; }k=j, \\    \frac 2{3 \cdot 5}, &  {\rm if\; } k \ne j. \end{cases}\label{lln_2} 
\end{align}
\end{lemma}
\begin{proof}
Since $\Lambda_n$ is symmetric under $\lambda \to - \lambda$ and ${\cal N}_n$ is even, we rewrite $R(n)$ as follows 
\begin{align*}
R(n)&=\frac{2}{{\cal N}_n} \sum_{\lambda \in  \Lambda_n/\pm} |a_{\lambda}|^4,
\end{align*}
where $ \Lambda_n/\pm$ denotes the representatives of the equivalence class of $\Lambda_n$ under $\lambda \to - \lambda$, $R(n)$ is then written in terms of a sum of independent and identically distributed random variables with $\mathbb{E}[ |a_{\lambda}|^4]=2$. The limit in \eqref{lln_1} follows from the Law of Large Numbers. To prove \eqref{lln_2}, we can apply again the Law of Large Numbers since  $\lambda_{(k)}^2  \lambda_{(j)}^2/n^2 \le 1$
\begin{align*}
R_{k,j}(n)&=\frac{1}{n^2 {\cal N}_n} \sum_{\lambda \in  \Lambda_n} \lambda_{(k)}^2  \lambda_{(j)}^2  |a_{\lambda}|^4= \frac{2}{n^2 {\cal N}_n} \sum_{\lambda \in  \Lambda_n/\pm } \lambda_{(k)}^2  \lambda_{(j)}^2  |a_{\lambda}|^4 \stackrel{\mathbb{P}}{\to} \lim_{n \to \infty} \frac{2}{n^2 {\cal N}_n} \sum_{\lambda \in  \Lambda_n } \lambda_{(k)}^2  \lambda_{(j)}^2.
\end{align*}
Formula \eqref{lln_2} follows from Lemma \ref{lambda_jk}.
\end{proof}

\begin{lemma} \label{X&Xjk} As $n \to \infty$, $n \not\equiv 0, 4, 7 \, (\rm{mod} \, 8)$
\begin{align*}
X(n), X_{k,k}(n), X_{k,k,j,j} (n) \stackrel{L^2}{\to} 0.
\end{align*}
\end{lemma}  
\begin{proof}
\begin{align*}
\mathbb{E}[|X(n)|^2]& =\frac{1}{{\cal N}^2_n} \mathbb{E}\Big[  \sum_{\cal{X}_n(4)} a_{\lambda_1}   a_{\lambda_2}  a_{\lambda_3}  a_{\lambda_4}   \sum_{\cal{X}_n(4)} a_{\lambda_1}   a_{\lambda_2}  a_{\lambda_3}  a_{\lambda_4}  \Big] =\frac{1}{{\cal N}^2_n} O(|\cal{X}_n(4)|),
\end{align*}
in view of Lemma \ref{JB}, we have 
\begin{align*}
\mathbb{E}[|X(n)|^2]& = \frac{1}{{\cal N}^2_n} O({\cal N}_n^{7/4+o(1)} )=O({\cal N}^{-1/4+o(1)}_n).
\end{align*}
Exactly the same holds for $X_{k,k}(n)$, $X_{k,k,j,j} (n)$, once we observe that $\lambda_{1,(k)}  \lambda_{2,(j)}/n \le 1$ and \newline $\lambda_{1,(k)}  \lambda_{2,(k)} \lambda_{3,(j)}  \lambda_{4,(j)}  /n^2 \le 1$.  
\end{proof}
Then we can write the fourth-order chaos of the nodal area ${\cal A}_n[4]$ as follows  
\begin{align} \label{11:21}
{\cal A}_n[4]
&=\frac{\sqrt n}{5  \sqrt 3 \, {\cal N}_n} \Big[ 4+ W^2(n) - 3 \sum_{j, k} W^2_{j,k}(n)  + o_{\mathbb{P}}(1)\Big].
\end{align}
Now we note that $W(n)= \sum_{k=1}^d W_{k,k}(n)$ so we can rewrite \eqref{11:21} as follows 
\begin{align*}
{\cal A}_n[4]&=\frac{\sqrt n}{5  \sqrt 3 \, {\cal N}_n} \Big[ 4- (W_{1,1}(n)-W_{2,2}(n))^2  - (W_{1,1}(n)-W_{3,3}(n))^2 - (W_{2,2}(n)-W_{3,3}(n))^2  \\
&\hspace{2cm}-6 ( W^2_{1,2}(n) + W^2_{1,3}(n) + W^2_{2,3}(n))  + o_{\mathbb{P}}(1)\Big].
\end{align*}

\noindent Let ${\cal W}(n)$ be the $7$-dimensional vector with components   
\begin{equation*}
{\cal W}(n)=(W_{1,1}(n), W_{1,2}(n), W_{1,3}(n), W_{2,2}(n),  W_{2,3}(n),  W_{3,3}(n))
\end{equation*}

\begin{lemma} \label{W&Wjk} As $n \to \infty$, 
${\cal W}(n) \stackrel{d}{\to} V$, 
where $V$ is a centred Gaussian vector with covariance matrix
$$\Sigma=\left(
\begin{matrix}
 \frac 2 5 & 0 & 0 &\frac  2 {3 \cdot 5} & 0 & \frac 2{3 \cdot 5} \\
0 & \frac  2 {3 \cdot 5} & 0 &0 & 0 &0 \\
0 & 0 &  \frac  2 {3 \cdot 5} &0 & 0 &0 \\
 \frac 2{3 \cdot 5} & 0 & 0 &\frac  2 {5} & 0 & \frac 2{3 \cdot 5} \\
0 &0& 0 &0 &  \frac  2 {3 \cdot 5}  &0 \\
 \frac 2{3 \cdot 5} & 0 & 0 &\frac  2 {3 \cdot 5} & 0 & \frac 2 5 
\end{matrix} \right).
$$
\end{lemma}
\begin{proof}
In Appendix \ref{appendixcov} we prove that the covariance matrix $\Sigma(n)$ of ${\cal W}(n)$ converges to $\Sigma$. Since for every fixed integer $n$ each component of ${\cal W}(n)$ belongs to the second Wiener chaos, in view of \cite[Theorem 6.2.3]{N&P}, the following two conditions are equivalent: 
\begin{enumerate}[1)]
\item ${\cal W}(n)$ converges in law to $V$;
\item each component of ${\cal W}(n)$ converges in distribution to a one-dimensional centred Gaussian random variable.\label{ii}
\end{enumerate}
We prove \ref{ii}): we observe that for $\lambda \in \Lambda_n/ \pm$ the random variables $|a_{\lambda}|^2$ are independent and identically distributed with mean and variance equal to $1$, we write 
\begin{align*}
W_{k,j}(n)&=2  \sum_{\lambda \in \Lambda_n/ \pm} (Q_{\lambda} - \mu_{\lambda}), 
\end{align*}
where 
$$Q_{\lambda}=\frac{\lambda_{(k)} \lambda_{(j)} }{n \sqrt{ {\cal N}_n}} |a_{\lambda}|^2, \hspace{1cm}\mu_{\lambda} =\frac{\lambda_{(k)} \lambda_{(j)} }{n \sqrt{ {\cal N}_n}},$$
and we note that the $Q_{\lambda}$ are independent random variables each with expected value $\mu_{\lambda}$ and variance $\mu^2_{\lambda}$, we also note that they have finite expected value and finite variance since $\lambda_{(k)}, \lambda_{(j)} \le \sqrt n$. Define
\begin{align*}
s^2_{\lambda}=\sum_{\lambda \in \Lambda_n/ \pm} {\rm Var}(Q_{\lambda}) 
= \begin{cases} 
\frac 1{2 \cdot 5} +O\left(\frac{1}{n^{1/28-o(1)}}\right), & k =j, \\
\frac 1{2 \cdot 3 \cdot 5} +O\left(\frac{1}{n^{1/28-o(1)}}\right), & k  \ne j. 
\end{cases} 
\end{align*}
We apply now the Lyapunov's condition so that we need to prove that 
\begin{align*}
\lim_{n \to \infty} \frac{1}{s^4_{\lambda}} \sum_{\lambda \in \Lambda_n/ \pm}  \mathbb{E} [|Q_{\lambda} - \mu_{\lambda}|^4]=0.
\end{align*}
To do that we first evaluate the $4$-th central moment of $Q_{\lambda}$; since $Q_{\lambda} \sim  \frac 1 2 \, \mu_{\lambda}\textrm{{\Large$\chi$}}_2$ where $\textrm{{\Large$\chi$}}_2$ is a chi-square with $2$ degrees of freedom, we need the moments: 
\begin{align*}
\mathbb{E}[ \textrm{{\Large$\chi$}}_2^m]=2^m \Gamma(m+1)
\end{align*}
so that 
\begin{align*}
\mathbb{E} [(Q_{\lambda} - \mu_{\lambda})^4] &= \mathbb{E}[Q^4_{\lambda} ] - 4 \mu_{\lambda}   \mathbb{E}[Q^3_{\lambda} ]+ 6  \mu^2_{\lambda}   \mathbb{E}[Q^2_{\lambda} ] - 4  \mu^3_{\lambda}  \mathbb{E}[Q_{\lambda} ] +  \mu^4_{\lambda}= \mu^4_{\lambda}.
\end{align*}
We finally note that 
\begin{align*}
0 \le \lim_{n \to \infty} \frac{1}{O(1)}  \sum_{\lambda \in \Lambda_n/ \pm}   \mu_{\lambda}^4 \le \lim_{n \to \infty} \frac{1}{O(1)}  \sum_{\lambda \in \Lambda_n/ \pm}  \frac{1}{{\cal N}^2_n}= \lim_{n \to \infty} \frac{1}{O(1)}   \frac{1}{2 {\cal N}_n}=0.
\end{align*}
So that 
\begin{align*}
\frac{1}{s_{\lambda}}W_{k,j}(n) \stackrel{d}{\to}  2 N(0,1)
\end{align*}
that is $W_{k,j}(n)$ converges in distribution to a centred Gaussian with variance $2/5$ if $k=j$ and $2/(3 \cdot 5)$ if $k \ne j$.
\end{proof}
\noindent Note that the covariance matrix $\Sigma$ is non-singular. The multidimensional CLT stated in Lemma \ref{W&Wjk} implies that 
\begin{align*}
& 4- (W_{1,1}(n)-W_{2,2}(n))^2  - (W_{1,1}(n)-W_{3,3}(n))^2 - (W_{2,2}(n)-W_{3,3}(n))^2 \\
&\;\; -6 ( W^2_{1,2}(n) + W^2_{1,3}(n) + W^2_{2,3}(n))  + o_{\mathbb{P}}(1) \\
& \stackrel{law}{\to}4- (V_{1,1}-V_{2,2})^2  - (V_{1,1}-V_{3,3})^2 - (V_{2,2}-V_{3,3})^2  -6 ( V^2_{1,2}+ V^2_{1,3} + V^2_{2,3}) \\
& = 4- \sum_{i=1}^3 X^2_{i}   -6  \sum_{i=1}^3 Y^2_{i},
\end{align*}
where $X=(X_1,X_2, X_3)$ and $Y=(Y_1, Y_2, Y_3)$ are two independent centred Gaussian vector with covariance matrices  
$$\Sigma_{X}
= \frac  {2^3}{3 \cdot 5} \left(
\begin{matrix}
1 &  \frac 1 2&- \frac 1 2 \\
\frac 1 2 &  1 &  \frac 1 2 \\
-\frac 1 2 &   \frac 1 2 &  1
\end{matrix} \right), \hspace{1cm}
\Sigma_{Y}=  \frac  {2}{3 \cdot 5} \left(
\begin{matrix}
1 &0&0\\
 0 & 1 &0\\
 0 &0&1
\end{matrix} \right),
$$
respectively. The covariance matrix $\Sigma_X$ is singular, hence we consider the transformation $X_1=Z_1$, $X_2=\alpha Z_1 + \beta Z_2$, $X_3=\gamma Z_1 + \delta Z_2$ where $Z=(Z_1, Z_2)$ is a centred Gaussian vector with covariance matrix $\Sigma_Z=\frac  {2^3}{3 \cdot 5} I_2$. We note that
\begin{align*}
&\text{Cov} (X_1, X_2)= \frac 1 2 \frac  {2^3}{3 \cdot 5}=\text{Cov}(Z_1, \alpha Z_1+\beta Z_2)= \frac  {2^3}{3 \cdot 5} \alpha,\\
&\text{Cov} (X_1, X_3)= - \frac 1 2 \frac  {2^3}{3 \cdot 5}=\text{Cov}(Z_1, \gamma Z_1+\delta Z_2)= \frac  {2^3}{3 \cdot 5} \gamma,\\
&\text{Cov} (X_2, X_3)=  \frac 1 2 \frac  {2^3}{3 \cdot 5}=\text{Cov}(\alpha Z_1+\beta Z_2, \gamma Z_1+\delta Z_2)=\frac  {2^3}{3 \cdot 5} (\alpha \gamma +\beta \delta),\\
&\text{Var} (X_2)=\frac  {2^3}{3 \cdot 5}=\text{Var} ( \alpha Z_1+\beta Z_2) = \frac  {2^3}{3 \cdot 5}(\alpha^2+\beta^2);
\end{align*} 
this implies that $\alpha=-\gamma=1/2$ and $\delta=\beta=\sqrt 3/2$. We write 
\begin{align} \label{13:20}
4- \sum_{i=1}^3 X^2_{i}   -6  \sum_{i=1}^3 Y^2_{i}&=4-\frac 6 4 \sum_{i=1}^2 Z^2_{i}   -6  \sum_{i=1}^3 Y^2_{i}
=4- 6   \frac{2}{3 \cdot 5}  \sum_{i=1}^5 U^2_i= 4-    \frac{4}{ 5} \textrm{{\Large$\chi$}};
\end{align}
where $U$ is a $5$-dimensional centred standard Gaussian vector and $\textrm{{\Large$\chi$}}$ is a central chi-square with $5$ degrees of freedom. 
\subsection{Proof of Theorem \ref{th}} In view of \eqref{13:20}, 
\begin{align*} 
\text{Var}({\cal A}_n[4]) =  \frac{n}{5^2 \cdot 3 \cdot {\cal N}^2_n} \frac{4^2}{5^2} \cdot 5 \cdot 2+o \left( \frac{n}{{\cal N}^2_n}\right) =
\frac{n}{ {\cal N}^2_n} \frac{2^5}{5^3 \cdot 3}+o \left( \frac{n}{{\cal N}^2_n}\right),
\end{align*}
and 
\begin{align*}
\frac{{\cal A}_n[4]}{\sqrt{ \text{Var}({\cal A}_n[4]) }} \stackrel{law}{\to} \frac{1}{\sqrt{5 \cdot 2}} \left(5-     \textrm{{\Large$\chi$}} \right).
\end{align*}
Theorem \ref{th} follows immediately by observing that ${\rm Var}({\cal A}_n) \sim {\rm Var}({\cal A}_n[4])$ and different chaotic components are orthogonal in $L^2$.

\appendix

\section{Chaotic expansion of ${\cal V}_n$: Proof Lemma \ref{caotic}} \label{caotic_A}

\subsection{Hermite expansion of $\frac{1}{2 \varepsilon} 1_{[-\varepsilon, \varepsilon]}(f_n(x))$}
We first expand the function $\frac{1}{2 \varepsilon} 1_{[-\varepsilon, \varepsilon]}(\cdot)$ into Hermite polynomials: using completeness and orthonormality of the set $\mathbb{H}$ in $L^2(\gamma)$, the following decomposition holds 
\begin{align*}
\frac{1}{2 \varepsilon} 1_{[-\varepsilon, \varepsilon]}(\cdot)= \sum_{k=0}^\infty \frac{1}{k!} \beta_k^{\varepsilon} H_k(\cdot),
\end{align*}
where 
\begin{align*}
\beta_0^{\varepsilon} = \frac{1}{2 \varepsilon} \int_{- \varepsilon}^{\varepsilon} \gamma(t) d t, \hspace{1cm} \beta_k^{\varepsilon} = \frac{1}{2 \varepsilon} \int_{- \varepsilon}^{\varepsilon} \gamma(t) H_k(t) d t, \;\; k \ge 1.
\end{align*}
From \eqref{hermite}, and observing that $H_k$ is an even function if $k$ is even, we easily obtain for $k =1, 2\dots$
\begin{align} \label{14giu}
\beta_{2 k-1}^{\varepsilon} =0,\hspace{0.5cm}  \beta_{2k}^{\varepsilon} &=-\frac{1}{ \varepsilon}   \gamma(\varepsilon) H_{2 k-1}( \varepsilon).  
\end{align}
For further details on the derivation of \eqref{14giu} see \cite[Lemma 3.4]{MPRW}

\subsection{Hermite expansion of $|| \nabla f_n(x) ||$} 
We first consider a standard Gaussian vector $Z$ in ${\mathbb R}^d$, the random variable $||Z||$ is square-integrable, so it can be expanded into a series of Hermite polynomials; we have 
\begin{align*} 
||Z|| = \sum_{p=0}^{\infty} \; \sum_{s \in \mathbb{N}^d: \; s_1+\cdots s_d=p} a(s) \prod_{j=1}^d H_{s_j} (Z_j),
\end{align*}
where 
\begin{align} \label{18:39}
a(s)= \frac{1}{s_1! \, s_2! \cdots s_d!} \frac{1}{(2 \pi)^{d/2}} \int_{\mathbb{R}^d} ||z|| \prod_{j=1}^d H_{s_j} (z_j) e^{-\frac{||z||^2}{2}} d z. 
\end{align}
We note that  $e^{-\frac{||z||^2}{2}}$ and  $||z||$ are even function of $z \in \mathbb{R}^d$, so $a(s)$ vanishes if one of the $s_j$ is odd. \\

We need to evaluate  the integrals $a(2s)$.  Following the proof in \cite[Lemma 3.5]{MPRW}, we introduce the Gaussian vector $U$ in $\mathbb{R}^d$ of independent random variables with unit variance and mean 
$\mu \in \mathbb{R}^d$, it is known that $||U||^2$ is distributed according to a non-central Chi-squared with probability density function 
\begin{align*}
f_{||U||^2} (t)=\sum_{i=0}^{\infty} \frac{1}{i!} e^{-\frac{||\mu||^2}{2}}  \frac{||\mu||^{2 i}}{2^i}  f_V(t), \hspace{1cm} f_V(t)= \frac{1}{2^{\frac{d+2 i}{2}}\Gamma(\frac{d+2 i}{2}) } t^{\frac{d+2i}{2}-1} e^{- \frac{t}{2}} 1_{\{t>0\}};
\end{align*}   
where $f_V$ is the density of a Chi-squared with $d+2 i$ degrees of freedom.  Therefore the probability density function of $||U||$ is $f_{||U||} (t) = 2 t f_{||U||^2} (t^2)$ and  
 its expectation is  
\begin{align} \label{18:41}
\mathbb{E}[||U||] &=  \int_0^{\infty} t f_{||U||} (t) dt 
=  e^{-\frac{||\mu||^2}{2}}   \sum_{i=0}^{\infty} \frac{1}{i!}||\mu||^{2 i} \; 2^{\frac 12 - i} \frac{ \Gamma(d/2+i+1/2) }{ \Gamma(d/2+i) }.   
\end{align}
We Taylor expand the exponential and we apply Newton's formula 
\begin{align*}
&e^{-\frac{||\mu||^2}{2}} = \sum_{k_1, \dots, k_d=0}^{\infty} \frac{(-1)^{k_1+ \dots +k_d}}{k_1! \cdots k_d!}
\frac{1}{2^{k_1+ \dots+k_d}} \mu_1^{2 k_1} \cdots \mu_d^{2 k_d}\\
&||\mu||^{2 i}= \sum_{j_1+\cdots+j_d=i} \binom{i}{j_1, \dots, j_d} \mu_1^{2 j_1} \cdots \mu_d^{2 j_d},
\end{align*}
to rewrite \eqref{18:41} as follows
\begin{align} \label{13:14}
\mathbb{E}[||U||]
&=    \sum_{l_1, \dots ,l_d=0}^{\infty}  \mu_1^{2 l_1} \cdots \mu_d^{2 l_d}
  \sum_{i=0}^{\infty} \frac{1}{i!\; 2^i }  \frac{ \sqrt 2 \Gamma(\frac{d}{2}+i+\frac{1}{2}) }{ \Gamma(\frac{d}{2}+i) } \\
& \times \sum_{\stackrel{j_1+\cdots+j_d=i}{j_1 \le l_1, \dots , j_d \le l_d}} \binom{i}{j_1,\dots , j_d}  \frac{(-1)^{l_1-j_1+\cdots +l_d-j_d}}{(l_1-j_1)! \cdots (l_d-j_d)! 2^{l_1-j_1+\cdots+l_d-j_d}}, \nonumber
\end{align}
where we set $l=k+j$.  
On the other hand, we can rewrite the Gaussian expectation $\mathbb{E}[||U||]$ as 
\begin{align*}
\mathbb{E}[||U||] =\frac{1}{(2 \pi)^{d/2}} \int_{\mathbb{R}^d} ||u|| e^{-\frac{||u-\mu||^2}{2}} d u
\end{align*}
and using the Hermite expansion of the exponential \cite[Proposition 1.4.2]{N&P} 
\begin{align*}e^{c x - \frac{c^2}{2}}=\sum_{l=0}^{\infty} \frac{c^l}{l!} H_l(x), \hspace{1cm} c \in \mathbb{R},\end{align*}
and we get 
\begin{align} \label{13:15}
\mathbb{E}[||U||] 
&= \frac{1}{(2 \pi)^{d/2}} \int_{\mathbb{R}^d} ||u|| \;  e^{-\frac{1}{2}||u||^2} \; \prod_{j=1}^d e^{-\frac{\mu_j^2}{2}+ \mu_j u_j}  \; d u \nonumber	 \\ 
  &=\frac{1}{(2 \pi)^{d/2}}   \sum_{l_1,\dots, l_d=0}^{\infty}    \frac{\mu_1^{l_1}   \cdots  \mu_d^{l_d}}{ l_1! \cdots l_d!} \int_{\mathbb{R}^d} ||u||  \;  e^{-\frac{1}{2} ||u||^2	} \;   \prod_{j=1}^d H_{l_j}(u_j)  d u \nonumber\\
 &= \frac{1}{(2 \pi)^{d/2}}  \sum_{l_1,\dots, l_d=0}^{\infty}      \frac{\mu_1^{2 l_1}   \cdots  \mu_d^{2 l_d}}{ (2 l_1)! \cdots (2 l_d)!} \int_{\mathbb{R}^d} ||u||   \;  e^{-\frac{1}{2} |u||^2	} \;      \prod_{j=1}^d 
 H_{2 l_j}(u_j)   d u,
\end{align}
since, if one of the $l_j$ is odd, then the integral vanishes. Combining  \eqref{13:14} and \eqref{13:15} we finally obtain 
\begin{align*}
&  \frac{1}{(2 \pi)^{d/2}} \frac{1}{ (2 l_1)! \cdots (2 l_d)!}  \int_{\mathbb{R}^d}     ||u||   \;  e^{-\frac{1}{2} ||u||^2} \;       \prod_{j=1}^d H_{2 l_j}(u_j)   d u \\
&=  \sum_{i=0}^{\infty} \frac{1}{i!\; 2^i }  \frac{ \sqrt 2 \Gamma(\frac d 2 +i+\frac 1 2) }{ \Gamma(\frac{d}{2}+i) }  \sum_{j_1+\cdots +j_d=i} \binom{i}{j_1, \dots , j_d}  \frac{(-1)^{l_1-j_1+\cdots+l_d-j_d}}{(l_1-j_1)! \cdots (l_d-j_d)! 2^{l_1-j_1+\cdots+l_d-j_d}},
\end{align*}
 and this implies that 
\begin{align*}
a(2 s)&= \sum_{i=0}^{\infty} \frac{1}{i!\; 2^i }  \frac{ \sqrt 2 \Gamma(\frac d 2 +i+ \frac 1 2) }{ \Gamma(\frac{d}{2}+i) }  \sum_{j_1+\cdots+j_d=i} \binom{i}{j_1, \dots , j_d}  \frac{(-1)^{s_1-j_1+\cdots+s_d-j_d}}{(s_1-j_1)! \cdots (s_d-j_d)! 2^{s_1-j_1+\cdots+s_d-j_d}}. 
\end{align*}\\

\noindent To obtain the Hermite expansion of $||\nabla f_n(x)||$, we note that 
\begin{align*}
 ||\nabla f_n(x)||&=2 \pi   \sqrt{ \frac{n }{d} } \, \left(  \sum_{j=1}^d  f^2_{n,j}(x) \right)^{1/2},
\end{align*} 
where $f_{n,j}(x)$ is the normalized derivative defined in \eqref{norm_d}, and that, for every fixed $x \in {\mathbb T}^d$, $f_n(x)$ and  $\partial_j f_n(x)$ are stochastically independent; then we have
\begin{align*}
\frac{1}{2 \varepsilon} 1_{[- \varepsilon, \varepsilon]} (f_n(x)) \; ||\nabla f_n(x)|| 
&=    \sum_{k=0}^{\infty} \frac{1}{(2k)!} \; \beta_{2 k}^{\varepsilon} \; H_{2k}(f_n(x))  2 \pi   \sqrt{ \frac{n }{d} }   \sum_{p=0}^{\infty} \sum_{s_1+\cdots+s_d=p}  a(2 s)  \prod_{j=1}^d H_{2 s_j} (f_{n,j}(x)),
\end{align*}
i.e. the projection onto each odd order chaos vanishes, whereas the projection onto the chaos of order $2 q$  for $q \ge 1$ is 
\begin{align*}
 {\cal V}_n^{\varepsilon}[2 q]=2 \pi   \sqrt{ \frac{n }{d} }    \sum_{p=0}^{q}    \frac{1}{(2q-2p)!} \; \beta_{2 q-2p}^{\varepsilon} \; H_{2q-2p}(f_n(x)) \sum_{s_1+\cdots+s_d=p} a(2 s)  \prod_{j=1}^d H_{2 s_j} (f_{n,j}(x)),
\end{align*}
where we set $2 q=2 k +2 p$ i.e. $2k=2q-2p$. 

\subsection{Evaluation of the coefficients $\beta_{2q-2p}$}
In view of the $L^2(\mathbb{P})$ convergence in \eqref{ttre}, the chaotic expansion of ${\cal V}_n$ follows by letting $\varepsilon \to 0$:  
$$\beta_{0}=\lim_{\varepsilon \to 0} \beta_{0}^{\varepsilon}=  \frac{1}{\sqrt{2 \pi}}, \hspace{1cm} \beta_{2k}=\lim_{\varepsilon \to 0} \beta_{2 k}^{\varepsilon}=- \lim_{\varepsilon \to 0} \frac{1}{ \varepsilon}   \gamma(\varepsilon) H_{2 k-1}( \varepsilon) =  \frac{1}{\sqrt{2 \pi}} H_{2 k }(0).$$

\noindent As noted in \cite[Section 3.2.2]{MPRW} we can interpret $\{\beta_{2 k}, k=0,1,2 \dots\}$ as the sequence of the coefficients appearing in the formal Hermite expansion of the Dirac mass $\delta_0$.\\

\section{Second and fourth order chaoses} \label{2&4_ch_A}

For $j,k =1, \dots, d$, we denote by $s(j)$ the vector in $\mathbb{R}^d$ with $j$-th component equal to $1$ and all the other components equal to $0$, and we denote by  $s(j,k)$ the vector in  $\mathbb{R}^d$ with $j$-th and $k$-th components equal to $1$ and all the other components equal to $0$. To evaluate the second and fourth order chaos we need the following lemma: 

\begin{proof}[Proof of Lemma \ref{aeb}]
From \eqref{18:30} we immediately derive $\beta_0$, $\beta_2$ and $\beta_4$. 
To evaluate $a(0)$ we note that, in view of \eqref{18:39}, $a(0)= {\mathbb E}[ ||Z||]$ where $Z$ is a standard Gaussian vector in $\mathbb{R}^d$,  then from \eqref{18:41} with $\mu=0$, we have 
\begin{align*}
a(0)&= {\mathbb E}[ ||Z||] 
=\sqrt 2 \; \frac{\Gamma(\frac d 2+ \frac 1 2)}{\Gamma(\frac d 2)}.
\end{align*}
Now let $\gamma(i)=\frac{1}{i! 2^i}  \frac{ \sqrt 2 \Gamma(\frac d 2 + i + \frac 1 2)}{\Gamma(\frac d 2 +i)}$, we have $\gamma(0)= \frac{ \sqrt 2 \Gamma(\frac{d+1}2)}{\Gamma(\frac d 2)}$, $\gamma(1)=\frac{1}{2}  \frac{ \sqrt 2 \Gamma(\frac{d+3}2)}{\Gamma(\frac{d+2}2)}$ and $\gamma(2)=\frac{1}{2^3}  \frac{ \sqrt 2 \Gamma(\frac{d+5}2)}{\Gamma(\frac{d+4}2)}$. To evaluate $a(2 s(k))$ we apply \eqref{18:48} and we obtain: 
\begin{align*}
a(2 s(k))&=\sum_{i=0}^{1} \gamma(i) \sum_{j_k=i} \binom{i}{j_k} \frac{(-1)^{q_k-j_k}}{(q_k - j_k)!} \frac{1}{2^{q_k - j_k}} 
 = -   \frac{\gamma(0)}{2} + \gamma(1)= \frac{1}{2 \sqrt 2} \frac{\Gamma( \frac d 2 + \frac 1 2)}{\Gamma(\frac d 2 +1)}.
\end{align*}
And similarly we derive $a(4 s(k))$ by observing that 
\begin{align*}
a(4 s(k))&=\sum_{i=0}^2 \gamma(i) \sum_{j_k=i} \binom{i}{j_k} \frac{(-1)^{2-j_k}}{(2-j_k)!} \frac{1}{2^{2-j_k}}=\frac{\gamma(0)}{2^3}  - \frac{ \gamma(1)}{2} + \gamma(2) 
=- \frac{\Gamma(\frac d 2 + \frac 1 2)}{2^4 \sqrt 2 \Gamma(\frac d 2 +2)}.
\end{align*} 
Finally,  for $j \ne k$,  we have that 
\begin{align*}
a(2 s(j,k))&= \sum_{i=0}^{2} \gamma(i) \sum_{j_1+j_2=i} \binom{i}{j_1, j_2} \frac{(-1)^{1-j_1+1-j_2}}{(1-j_1)! (1-j_2)!  } \frac{1}{2^{1-j_1+1-j_2}}
=  \frac{\gamma(0)}{2^2} -  \gamma(1)  +2 \gamma(2)\\
&= - \frac{1}{2^3 \sqrt 2} \frac{\Gamma(\frac d 2 + \frac 1 2)}{\Gamma(\frac d 2 +2)}.
\end{align*}
\end{proof}

\begin{proof}[Proof of Lemma \ref{14:20}]
We apply the orthogonality relations of exponentials: 
 \begin{align*}
 \int_{\mathbb{T}^d} e(\langle \mu, x \rangle) d x =  \begin{cases} 1& \mu=0 \\ 0& \mu \ne 0. \end{cases} 
 \end{align*}
Note that $H_2(x)=x^2-1$, so we have 
\begin{align*}
 \int_{\mathbb{T}^d}  H_{2}(f_n(x))  d x&=  \int_{\mathbb{T}^d}  (f^2_n(x)-1)  d x= \frac{1}{{\cal N}_n} \sum_{\lambda_1, \lambda_2} a_{\lambda_1} a_{\lambda_2} \int_{\mathbb{T}^d} e (\langle \lambda_1, x \rangle)  e (\langle \lambda_2, x \rangle)  d x -1\\
 &= \frac{1}{{\cal N}_n} \sum_{\lambda_1+ \lambda_2=0} a_{\lambda_1} a_{\lambda_2} -1= \frac{1}{{\cal N}_n} \sum_{\lambda} a_{\lambda} \overline{a}_{\lambda}  -1= \frac{1}{{\cal N}_n} \sum_{\lambda} |a_{\lambda}|^2   -1\\
 &= \frac{1}{{\cal N}_n} \sum_{\lambda} \left( |a_{\lambda}|^2   -1 \right).
\end{align*}
\noindent For $k=1, \dots, d$,  
\begin{align*}
 \int_{\mathbb{T}^d}  H_{2}(f_{n,k}(x))  d x&=  \int_{\mathbb{T}^d}  (f^2_{n,k}(x)-1)  d x\\
 &= -\frac{d}{n   {\cal N}_n} \sum_{\lambda_1, \lambda_2} a_{\lambda_1} a_{\lambda_2} \lambda_{1,(k)} \,  \lambda_{2,(k)} \int_{\mathbb{T}^d} e (\langle \lambda_1, x \rangle)  e (\langle \lambda_2 , x \rangle)  d x -1\\
 &=- \frac{ d}{  n   {\cal N}_n} \sum_{\lambda_1 + \lambda_2=0} a_{\lambda_1} a_{\lambda_2} \lambda_{1,(k)}  \lambda_{2,(k)}   -1= \frac{d }{  n {\cal N}_n} \sum_{\lambda} |a_{\lambda}|^2 \lambda^2_{(k)}     -1 \\
  &= \frac{d }{  n {\cal N}_n} \sum_{\lambda}  \lambda^2_{(k)}  ( |a_{\lambda}|^2   -1 ). 
\end{align*}
where in the last step we applied \eqref{RW}. 
\end{proof}

\begin{proof}[Proof of Lemma \ref{inc&esc}]
Formula \eqref{1st} and formula \eqref{3rd} follow immediately form \eqref{M}. To prove \eqref{2nd} we note that 
\begin{align*}
 \sum_{ {\cal C}(4)} \lambda_{1,(k)}\, \lambda_{2,(k)} \, a_{\lambda_1} a_{\lambda_2} a_{\lambda_3} a_{\lambda_4}& = - \sum_{\lambda_1, \lambda_3}  \lambda_{1,(k)}^2 |a_{\lambda_1}|^2  |a_{\lambda_3}|^2 + 2 \sum_{\lambda_1, \lambda_2}  \lambda_{1,(k)}  \lambda_{2,(k)}  |a_{\lambda_1}|^2  |a_{\lambda_2}|^2 \\
 &\;\;+ \sum_{\lambda} \lambda^2_{(k)}  |a_{\lambda}|^4   + \sum_{{\cal X}(4)} \lambda_{1,(k)}\, \lambda_{2,(k)} \, a_{\lambda_1} a_{\lambda_2} a_{\lambda_3} a_{\lambda_4}
\end{align*}
where the second term cancels since $a_{-\lambda}=\overline{a}_{\lambda}$ and then 
\begin{align*}
\sum_{\lambda} \lambda_{(k)} |a_{\lambda}|^2 =0. 
\end{align*}
To prove the identity \eqref{4th} we apply again \eqref{M} to get
\begin{align*}
& \sum_{{\cal C}(4)}   \lambda_{1,(k)}  \lambda_{2,(k)} \lambda_{3,(j)} \lambda_{4,(j)} a_{\lambda_1} a_{\lambda_2} a_{\lambda_3} a_{\lambda_4} \\
&= \sum_{\lambda_1, \lambda_3} \lambda^2_{1,(k)}  \lambda^2_{3,(j)}  |a_{\lambda_1}|^2 |a_{\lambda_3}|^2 + 2 \sum_{\lambda_1, \lambda_2} \lambda_{1,(k)}  \lambda_{2,(k)}  \lambda_{1,(j)}  \lambda_{2,(j)}  |a_{\lambda_1}|^2 |a_{\lambda_2}|^2-3 \sum_{\lambda} \lambda^2_{(k)} \lambda^2_{(j)} |a_{\lambda}|^4  \\
 &\;\;+ \sum_{{\cal X}(4)}   \lambda_{1,(k)}  \lambda_{2,(k)} \lambda_{3,(j)} \lambda_{4,(j)} a_{\lambda_1} a_{\lambda_2} a_{\lambda_3} a_{\lambda_4}
\end{align*}
where we note that, in view of \eqref{RW}, we can write  
\begin{align*}
\sum_{\lambda_1, \lambda_2} \lambda_{1,(k)}  \lambda_{2,(k)}  \lambda_{1,(j)}  \lambda_{2,(j)}  |a_{\lambda_1}|^2 |a_{\lambda_2}|^2= \left[ \sum_{\lambda} \lambda_{(j)} \lambda_{(k)} (|a_{\lambda}|^2-1) \right]^2.
\end{align*}
\end{proof}

\begin{proof}[Proof of Lemma \ref{15:10}]
We use $H_4(x)=x^4-6 x^2+3$ and formula \eqref{11st} to write
\begin{align*}
\int_{\mathbb{T}^d}  H_{4}(f_n(x))dx &= \int_{\mathbb{T}^d} f^4_n(x)dx - 6  \int_{\mathbb{T}^d} f^2_n(x)dx +3\\
&=  \frac{1}{{\cal N}^2_n} \sum_{{\cal C}(4)} a_{\lambda_1} a_{\lambda_2} a_{\lambda_3} a_{\lambda_3}-     \frac{6}{{\cal N}_n} \sum_{\lambda}  |a_{\lambda}|^2   + 3
\end{align*}
in view of formula \eqref{1st} we have 
\begin{align*}
\int_{\mathbb{T}^d}  H_{4}(f_n(x))dx&=  \frac{3}{{\cal N}^2_n}  \sum_{\lambda_1, \lambda_2} |a_{\lambda_1}|^2  |a_{\lambda_2}|^2 -  \frac{3}{{\cal N}^2_n} \sum_{\lambda} |a_{\lambda}|^4 +  \frac{1}{{\cal N}^2_n} \sum_{{\cal \chi(4)}} a_{\lambda_{1}} a_{\lambda_{2}} a_{\lambda_{3}} a_{\lambda_{4}}-     \frac{6}{{\cal N}_n} \sum_{\lambda}  |a_{\lambda}|^2 +3\\
&= \frac{3}{{\cal N}^2_n} \left[   \sum_{\lambda} ( |a_{\lambda}|^2  -1) \right]^2  -  \frac{3}{{\cal N}^2_n} \sum_{\lambda} |a_{\lambda}|^4  +   \frac{1}{{\cal N}^2_n} \sum_{{\cal X}(4)} a_{\lambda_{1}} a_{\lambda_{2}} a_{\lambda_{3}} a_{\lambda_{4}}.
\end{align*}
\noindent We evaluate now 
\begin{align*}
&\sum_{k=1}^d \int_{\mathbb{T}^d}   H_{2}(f_n(x))  H_{2}(f_{n,k}(x))  dx\\
& = \sum_{k=1}^d \int_{\mathbb{T}^d}   \left( f_n^2(x) -1 \right) \left( f^2_{n,k}(x) -1 \right) d x \\
&= \sum_{k=1}^d \left\{  \int_{\mathbb{T}^d} f_n^2(x)   f^2_{n,k}(x) d x  -   \int_{\mathbb{T}^d} f^2_n(x)  d x   -    \int_{\mathbb{T}^d} f^2_{n,k}(x)  d x +1 \right\} \\
&=  \sum_{k=1}^d \left\{ -\frac{ d}{n\, {\cal N}^2_n}  \sum_{ {\cal C}(4)} \lambda_{1,(k)}\, \lambda_{2,(k)} \, a_{\lambda_1} a_{\lambda_2} a_{\lambda_3} a_{\lambda_4}  -   
\frac{1}{{\cal N}_n} \sum_{\lambda} |a_{\lambda}|^2 -   \frac{d}{n {\cal N}_n} \sum_{\lambda}\lambda^2_{(k)} |a_{\lambda}|^2 +1 \right\}\\
&= -\frac{d}{n\, {\cal N}^2_n} \sum_{k=1}^d \sum_{ {\cal C}(4)} \lambda_{1,(k)}\, \lambda_{2,(k)} \, a_{\lambda_1} a_{\lambda_2} a_{\lambda_3} a_{\lambda_4}  - \frac{2 d}{ {\cal N}_n} \sum_{\lambda} |a_{\lambda}|^2 +d
\end{align*}
where we applied \eqref{11st} and \eqref{22nd} and in the last step we use the fact that $\sum_k \lambda^2_{(k)}=n$; from formula \eqref{2nd} and formula \eqref{RW} we have 
\begin{align*}
&\sum_{k=1}^d \int_{\mathbb{T}^d}   H_{2}(f_n(x))  H_{2}(f_{n,k}(x))  dx\\
& =  \frac{d}{{\cal N}^2_n}  \sum_{\lambda_1, \lambda_2}  |a_{\lambda_1}|^2  |a_{\lambda_2}|^2  - \frac{d}{{\cal N}^2_n}  \sum_{\lambda} |a_{\lambda}|^4 - \frac{d}{n {\cal N}^2_n} \sum_{k=1}^d \sum_{{\cal X}(4)} \lambda_{1,(k)}\, \lambda_{2,(k)} \, a_{\lambda_1} a_{\lambda_2} a_{\lambda_3} a_{\lambda_4}- \frac{2 d}{ {\cal N}_n} \sum_{\lambda} |a_{\lambda}|^2 +d\\
&= \frac{d}{{\cal N}^2_n} \left[ \sum_{\lambda} (|a_{\lambda}|^2 -1) \right]^2    - \frac{d}{{\cal N}^2_n}  \sum_{\lambda} |a_{\lambda}|^4 - \frac{d}{n {\cal N}^2_n} \sum_{k=1}^d \sum_{{\cal X}(4)} \lambda_{1,(k)}\, \lambda_{2,(k)} \, a_{\lambda_1} a_{\lambda_2} a_{\lambda_3} a_{\lambda_4}. 
\end{align*}
We also have that 
\begin{align*}
& \int_{\mathbb{T}^d} H_4(f_{n,k}(x)) d x\\
 &=  \frac{d^2}{n^2 {\cal N}^2_n}  \sum_{{\cal C}(4)} \lambda_{1,(k)} \lambda_{2,(k)} \lambda_{3,(k)} \lambda_{4,(k)}
a_{\lambda_1} a_{\lambda_2} a_{\lambda_3} a_{\lambda_4} - 6  \int_{\mathbb{T}^d} f^2_{n,k}(x) d x + 3 \\
&=3  \frac{d^2}{n^2 {\cal N}^2_n}     \sum_{\lambda_1, \lambda_2} \lambda^2_{1,(k)}   \lambda^2_{2,(k)} |a_{\lambda_1}|^2  |a_{\lambda_2}|^2 - 3  \frac{d^2}{n^2 {\cal N}^2_n}   \sum_{\lambda} \lambda^4_{(k)} |a_{\lambda}|^4 \\
&\;\;+\frac{d^2}{n^2 {\cal N}^2_n}  \sum_{{\cal X}(4)}    \lambda_{1,(k)} \lambda_{2,(k)} \lambda_{3,(k)} \lambda_{4,(k)} a_{\lambda_1} a_{\lambda_2} a_{\lambda_3} a_{\lambda_4} -  \frac{6}{ {\cal N}_n}  \sum_{\lambda} |a_{\lambda}|^2  + 3 \\
&= \frac{3 d^2}{n^2 {\cal N}^2_n}  \left[  \sum_{\lambda} \lambda^2_{(k)} (|a_{\lambda}|^2 -1) \right]^2 - \frac{3 d^2}{n^2 {\cal N}^2_n}  \sum_{\lambda} \lambda^4_{(k)} |a_{\lambda}|^4  +\frac{d^2}{n^2 {\cal N}^2_n}    \sum_{{\cal X}(4)}    \lambda_{1,(k)} \lambda_{2,(k)} \lambda_{3,(k)} \lambda_{4,(k)} a_{\lambda_1} a_{\lambda_2} a_{\lambda_3} a_{\lambda_4}.
\end{align*}
And finally 
\begin{align*}
& \int_{\mathbb{T}^d} H_2(f_{n,j}(x) )  H_2(f_{n,k}(x) ) d x \\
&= \int_{\mathbb{T}^d}  \left( f^2_{n,j}(x) -1 \right)  \left( f^2_{n,k}(x) -1 \right)   d x \\
&=  \int_{\mathbb{T}^d}  \left( f^2_{n,j}(x) f^2_{n,k}(x) - f^2_{n,j}(x) -f^2_{n,k}(x)  +1 \right)   d x \\
&=  \frac{d^2}{n^2 {\cal N}^2_n}    \sum_{{\cal C}(4)}         \lambda_{1,(k)}  \lambda_{2,(k)} \lambda_{3,(j)} \lambda_{4,(j)} a_{\lambda_1} a_{\lambda_2} a_{\lambda_3} a_{\lambda_4}  - \frac{d}{n {\cal N}_n} \sum_{\lambda} \lambda^2_{(j)} |a_{\lambda}|^2 - \frac{d}{n {\cal N}_n} \sum_{\lambda} \lambda^2_{(k)} |a_{\lambda}|^2 +1,
\end{align*}
in view of Lemma \ref{inc&esc}, formula \eqref{4th},
\begin{align*}
& \int_{\mathbb{T}^d} H_2(f_{n,j}(x) )  H_2(f_{n,k}(x) ) d x \\
&=   \frac{d^2}{n^2 {\cal N}^2_n}   \sum_{\lambda_1, \lambda_2} \lambda^2_{1,(k)}  \lambda^2_{2,(j)}  |a_{\lambda_1}|^2 |a_{\lambda_2}|^2 +   \frac{2 d^2}{n^2 {\cal N}^2_n}   \Big[ \sum_{\lambda} \lambda_{(j)} \lambda_{(k)} (|a_{\lambda}|^2-1) \Big]^2\\
 &\;\;-    \frac{3 d^2}{n^2 {\cal N}^2_n}  \sum_{\lambda} \lambda^2_{(k)} \lambda^2_{(j)} |a_{\lambda}|^4+  \frac{d^2}{n^2 {\cal N}^2_n}  \sum_{{\cal X}(4)}   \lambda_{1,(k)}  \lambda_{2,(k)} \lambda_{3,(j)} \lambda_{4,(j)} a_{\lambda_1} a_{\lambda_2} a_{\lambda_3} a_{\lambda_4} \\
 &\;\; - \frac{d}{n {\cal N}_n} \sum_{\lambda} \lambda^2_{(j)} |a_{\lambda}|^2 - \frac{d}{n {\cal N}_n} \sum_{\lambda} \lambda^2_{(k)} |a_{\lambda}|^2 +1  \\
 &=   \frac{d^2}{n^2 {\cal N}^2_n}   \sum_{\lambda_1, \lambda_2} \lambda^2_{1,(k)}  \lambda^2_{2,(j)} ( |a_{\lambda_1}|^2-1) (|a_{\lambda_2}|^2-1) +   \frac{2 d^2}{n^2 {\cal N}^2_n}   \Big[ \sum_{\lambda} \lambda_{(j)} \lambda_{(k)} (|a_{\lambda}|^2-1) \Big]^2\\ 
 &\;\;-    \frac{ 3 d^2}{n^2 {\cal N}^2_n}  \sum_{\lambda} \lambda^2_{(k)} \lambda^2_{(j)} |a_{\lambda}|^4+  \frac{d^2}{n^2 {\cal N}^2_n}  \sum_{{\cal X}(4)}   \lambda_{1,(k)}  \lambda_{2,(k)} \lambda_{3,(j)} \lambda_{4,(j)} a_{\lambda_1} a_{\lambda_2} a_{\lambda_3} a_{\lambda_4}\\
 &= \frac{d^2}{ {\cal N}_n}  W_{k,k}(n) W_{j,j}(n)+   \frac{2 d^2}{{\cal N}_n}  W^2_{j,k} (n)-    \frac{ 3 d^2}{ {\cal N}_n}  R_{k,j}(n) +  \frac{d^2}{n^2 {\cal N}^2_n} X_{k,k,j,j}(n);
\end{align*}
now we note that
\begin{align*}
W_{kk} (n) \sum_{j \ne k} W_{jj} (n)&= \frac{1}{n^2 {\cal N}_n} \sum_{\lambda_1 \in \Lambda_n} \lambda^2_{1,(k)}  (|a_{\lambda}|^2-1) \sum_{j \ne k} \sum_{\lambda_2 \in \Lambda_n} \lambda^2_{2,(j)}  (|a_{\lambda}|^2-1)\\
&=  \frac{1}{n^2 {\cal N}_n} \sum_{\lambda_1 \in \Lambda_n} \lambda^2_{1,(k)}  (|a_{\lambda}|^2-1)  \sum_{\lambda_2 \in \Lambda_n} (n- \lambda^2_{2,(k)})  (|a_{\lambda}|^2-1)\\
&= W(n) W_{k,k}(n) - W^2_{k,k}(n)
\end{align*}
so that 
\begin{align*}
\sum_{k=1}^d W_{kk} (n) \sum_{j \ne k} W_{jj} (n)&=  W(n) \sum_{k=1}^d  W_{k,k}(n) - \sum_{k=1}^d W^2_{k,k}(n)= W^2(n) - \sum_{k=1}^d W^2_{k,k}(n).
\end{align*}
\end{proof}

\section{Covariance matrices $\Sigma(n)$ and $\Sigma$} \label{appendixcov}
In this section we compute the covariance matrix $\Sigma(n)$ and its limiting matrix $\Sigma$; we apply Lemma \ref{equid}, Lemma \ref{grouplemma}, equation \eqref{inv} and we use
  \begin{align*}
  \mathbb{E}[ (|a_{\lambda_1}|^2-1) (|a_{\lambda_2}|^2-1)]= \begin{cases} 1, &  {\rm if} \;  \lambda_1=\pm \lambda_2, \\ 0, & {\rm  otherwise}.   \end{cases}
  \end{align*}
We note that 
\begin{align*}
&\mathbb{E}[W_{k,l}(n) W_{j,m}(n)]\\
&= \frac{1}{n^2 {\cal N}_n} \sum_{\lambda_1, \lambda_2} \lambda_{1,(k)} \lambda_{1,(l)}  
\lambda_{2,(j)} \lambda_{2,(m)}   \mathbb{E}[ (|a_{\lambda_1}|^2-1)  (|a_{\lambda_2}|^2-1)]\\
&= \frac{1}{n^2 {\cal N}_n} \sum_{\lambda} \{  \lambda_{(k)} \lambda_{(l)}  
\lambda_{(j)} \lambda_{(m)}   \mathbb{E}[ (|a_{\lambda}|^2-1)^2] +  \lambda_{(k)} \lambda_{(l)}  
(- \lambda_{(j)}) (- \lambda_{(m)})   \mathbb{E}[ (|a_{\lambda}|^2-1)^2]  \} \\
&= \frac{2}{n^2 {\cal N}_n} \sum_{\lambda}  \lambda_{(k)} \lambda_{(l)}  
\lambda_{(j)} \lambda_{(m)}   \mathbb{E}[ (|a_{\lambda}|^2-1)^2] \\
&= \frac{2}{n^2 {\cal N}_n} \sum_{\lambda}  \lambda_{(k)} \lambda_{(l)} \lambda_{(j)}   \lambda_{(m)}  \\
&= \begin{cases} 
 \frac{2}{3 \cdot 5}+O\left( \frac{1}{n^{1/28 -o(1)}}\right), & k=l, j=m, k \ne j {\; \rm or\;}  k=j,  l=m,  k \ne l {\; \rm or\;}  k=m,  l=j, k \ne l, \\ 
\frac{2}{ 5}+O\left( \frac{1}{n^{1/28 -o(1)}}\right), & k=l=j=m,\\
0, &  {\rm otherwise}.
 \end{cases}
\end{align*}


\begin{thebibliography}{99}

\bibitem{B&M}
J.~Benatar and R. W.~Maffiucci. 
Random waves on $\mathbb{T}^3$: nodal area variance and lattice points correlations.  To appear in {\it Int. Math. Res. Notices.}
\href{https://arxiv.org/pdf/1708.07015.pdf}{\it arXiv:1708.07015} 

\bibitem{MB}
M. V.~Berry.
Statistics of nodal lines and points in chaotic quantum billiards: perimeter corrections, fluctuations, curvature. 
{\it Journal of Physics A: Mathematical and General}, {\bf 35}, no.13, (2002). 

\bibitem{BR&S}
Y.~Bourgain, P.~Sarnak, and Z.~Rudnik.
Local statistics of lattice points on the sphere. 
{\it Contemporary Mathematics}, {\bf 661}, (2016). 

\bibitem{CM}
V.~Cammarota and D.~Marinucci. 
A Quantitative Central Limit Theorem for the
Euler-Poincar\'e Characteristic of Random
Spherical Eigenfunctions. 
\href{https://arxiv.org/pdf/1603.09588.pdf}{\it arXiv:1603.09588} 


\bibitem{cheng}
S.-Y.~Cheng.
Eigenfunctions and nodal sets. 
{\it Comm. Math. Helv.}, {\bf 51}(1), (1976), 43-55. 

\bibitem{DNPR}
F.~Dalmao, I.~Nourdin, G.~Peccati and M.~Rossi. 
Phase Singularities in Complex Arithmetic Random Waves. 
\href{https://arxiv.org/pdf/1608.05631.pdf}{\it arXiv: 1608.05631} 


\bibitem{D&F}
H.~Donnelly and C.~Fefferman.
Nodal sets of eigenfunctions on Riemannian manifolds. 
{\it Invent. Math}, {\bf 93}, (1988), 161-183. 

\bibitem{D}
W.~Duke. 
Hyperbolic distribution problems and half-integral weight Maass forms. 
{\it Invent. Math}, {\bf 92} no. 1, (1988), 73-90. 

\bibitem{D3}
W.~Duke and R.~Schulze-Pillot.
Representation of integers by positive ternary quadratic forms and equidistribution of lattice points on ellipsoids. 
{\it Invent. Math.}, {\bf 99}, no. 1, (1990), 49Ð57.


\bibitem{D2}
W.~Duke.  
{\it An Introduction to the Linnik Problems}. 
In: Granville A., Rudnick Z. (eds) Equidistribution in Number Theory, An Introduction. 
NATO Science Series, vol 237. Springer, Dordrecht (2007). 

\bibitem{G&F}
E. P.~Golubeva and O. M.~Fomenko. 
Asymptotic equidistribution of integral points on the three-dimensional sphere. 
{\it Zap. Nauchn, Sem. Leningrad Otdel. Math. Inst. Steklov}, {\bf 160}, (1987), 54-71.  

\bibitem{H&W}
G. H.~Hardy and E. M. Wright. 
{ \it An Introduction to the Theory of Numbers}.
Oxford University Press,  Oxford (1960).

\bibitem{I}
H.~Iwaniec. 
Fourier coefficients of modular forms of half-integral weight. 
{\it Invent. Math.}, {\bf 87} (1987), no. 2, 385Ð401.

\bibitem{I&K}
H.~Iwaniec and E.~Kowalski. 
{\it Analytic Number Theory}. 
Colloquium Publications. Volume 53 (2004)


\bibitem{K&L}
M. F.~Kratz and J. R.~ Le\'on. 
Central Limit Theorems for Level Functionals of Stationary Gaussian Processes and Fields. 
{\it Journal of Theoretical Probability}, {\bf 14}, no. 3, (2001), 639-672. 

\bibitem{KKW}
M.~Krishnapur, P.~Kulberg and I.~Wigman. 
Nodal length fluctuations for arithmetic random waves. 
{\it Ann. of Math.}, {\bf 2}, no. 177, (2013) 699-737.

 
\bibitem{La}
E.~Landau.
Uber die einteilung der positiven zahlen nach vier klassen nach der mindestzahl der zu ihrer addition zusammensetzung erforderlichen quadrate. 
{\it Archiv. der Mathematik und Physics III}, (1908).

\bibitem{L}
Y. V.~Linnik. 
{\it Ergodic Properties of Algebraic Fields}. Ergebnisse der Mathematik und ihrer Grenzgebiete, Band 45, Springer-Verlag New York Inc., New York (1968).

\bibitem{logunov}
A.~Logunov. 
Nodal sets of Laplace eigenfunctions: proof of Nadirashvili's conjecture and of the lower bound in Yau's conjecture.
\href{https://arxiv.org/pdf/1605.02589.pdf}{\it arXiv:1605.02589} 
 

\bibitem{M}
A. V.~Malyshev. 
On representations of integers by positive quadratic forms. 
{\it Trudy V. A. Steklov Math. Inst. AN SSSR}, {\bf 65} (1962), 1-212. 

\bibitem{MPRW}
D.~Marinucci, G.~Peccati, M.~Rossi and I.~Wigman. 
Non-universality of nodal length distribution for arithmetic random waves. 
{\it Geom. Funct. Anal.}, {\bf 26}, no. 3, (2016) 926-960. 


\bibitem{N&P}
I.~Nourdin and G.~Peccati. 
{\it Normal Approximations with Malliavin Calculus. From Stein's Methods to Universality}. Cambridge University Press (2012).  

\bibitem{P&S}
A.~Palczewski, J.~Schneider and A. V.~Bobylen. 
A consistency result for a discrete-velocity model of the Boltzmann equation. 
{\it SIAM J. Numer. Anal.}, {\bf 34}(5), (1997), 1865-1883.   

\bibitem{P&R}
G.~Peccati and M.~Rossi. 
Quantitative limit theorems for local functionals of arithmetic random waves. 
\href{https://arxiv.org/pdf/1702.03765.pdf}{\it arXiv:1702.03765} 
  

\bibitem{P}
C.~Pommerenke. 
\"Uber die Gleichverteilung von Gitterpunkten auf $m$-dimensionalen Ellipsoiden. 
{\it Acta Arith.} {\bf 5} (1959), 227-257. 


\bibitem{cheng1}
S.-T. ~Yau.
Survey on partial differential equations in differential geometry. 
{\it Ann. Math. Studies}, {\bf 102}, (1982), 3-70. 

\bibitem{cheng2}
S.-T.~Yau.
Open problems in geometry. 
{\it Proc. Symp. Pure Math}, {\bf 54}, (1993), 1-28. 


\bibitem{R&W}
Z.~Rudnick and I.~Wigman.
On the volume of nodal sets for eigenfunctios of the Laplacian on the torus.
{\it Ann. Henri Poincar\'e}, {\bf 9}(1), (2008), 109-130. 



\bibitem{W}
N.~Wiener. 
The homogeneous chaos. 
{\it Amer. J. Math.}, {\bf 60},  (1938), 879-936.






\end{thebibliography}
\end{document}